\documentclass{amsart}
\usepackage{amsmath,amssymb,enumerate}

\title[How incomputable is the separable Hahn-Banach theorem?]{How incomputable\\
    is the separable Hahn-Banach theorem?}

\author{Guido Gherardi}
    \address{Dipartimento di Filosofia,
    Universit\`{a} di Bologna,
    via Zamboni 38,
    40126 Bologna,
    Italy}
\email{guido.gherardi@unibo.it}

\author{Alberto Marcone}
    \address{Dipartimento di Matematica e Informatica,
    Universit\`{a} di Udine,
    viale delle Scienze 206,
    33100 Udine,
    Italy}
\email{alberto.marcone@dimi.uniud.it}

\newtheorem{theorem}{Theorem}[section]
\newtheorem{lemma}[theorem]{Lemma}

\newtheorem{cor}[theorem]{Corollary}

\theoremstyle{definition}
\newtheorem{definition}[theorem]{Definition}
\newtheorem{remark}[theorem]{Remark}

\newtheorem{ex}[theorem]{Example}

\numberwithin{equation}{section}

\newcommand{\set}[2]{\{\,{#1}\mid{#2}\,\}}
\newcommand{\SET}[2]{\left\{\,{#1}\mid{#2}\,\right\}}
\newcommand{\la}{\left\langle}
\newcommand{\ra}{\right\rangle}
\newcommand{\conc}{{}^\smallfrown}
\newcommand{\toto}{\rightrightarrows}
\newcommand{\sbsq}{{\subseteq}}
\newcommand{\dom}{\operatorname{dom}}
\newcommand{\ran}{\operatorname{ran}}

\newcommand{\N}{{\mathbb N}}
\newcommand{\Q}{{\mathbb Q}}
\newcommand{\R}{{\mathbb R}}
\newcommand{\Can}{\ensuremath{{2^\N}}}
\newcommand{\Bai}{\ensuremath{{\N^\N}}}
\newcommand{\LL}{\ensuremath{\mathcal{L}_2}}

\newcommand{\mcA}{\mathcal{A}}

\newcommand{\mcK}{\mathcal{K}}

\newcommand{\cred}{\leqslant_c}
\newcommand{\ncred}{\nleqslant_c}
\newcommand{\creds}{<_c}
\newcommand{\ceq}{\cong_c}

\def\a{\alpha}
\def\b{\beta}
\def\g{\gamma}
\def\d{\delta}

\def\e{\eta}
\def\i{\iota}
\def\k{\kappa}
\def\l{\lambda}
\def\m{\mu}

\def\p{\psi}
\def\s{\sigma}
\def\r{\rho}

\def\Otto{\Longleftrightarrow}

\def\fa{\forall}

\def\eps{\emptyset}

\newcommand{\PI}[2]{\ensuremath{\boldsymbol\Pi^{#1}_{#2}}}
\newcommand{\SI}[2]{\ensuremath{\boldsymbol\Sigma^{#1}_{#2}}}
\newcommand{\DE}[2]{\ensuremath{\boldsymbol\Delta^{#1}_{#2}}}

\newcommand{\system}[1]{\mbox{\fontfamily{cmss}\fontshape{n}\fontseries{m}%
    \selectfont#1}}
\newcommand{\RCA}{\system{RCA}\ensuremath{_0}}
\newcommand{\WKL}{\system{WKL}\ensuremath{_0}}
\newcommand{\ACA}{\system{ACA}\ensuremath{_0}}
\newcommand{\ATR}{\system{ATR}\ensuremath{_0}}
\newcommand{\PCA}{\system{$\PI11$-CA}\ensuremath{_0}}

\newcommand{\Sep}{\ensuremath{\operatorname{\mathbf{Sep}}}}
\newcommand{\Seq}{\ensuremath{\N^{<\N}}}
\newcommand{\Seqd}{\ensuremath{2^{<\N}}}
\newcommand{\Path}{\ensuremath{\operatorname{\mathbf{Path}_2}}}
\newcommand{\InfTr}{\ensuremath{\operatorname{\mathbf{InfTr}_2}}}
\newcommand{\PathB}{\ensuremath{\operatorname{\mathbf{Path}_B}}}
\newcommand{\InfTrB}{\ensuremath{\operatorname{\mathbf{InfTr}_B}}}
\newcommand{\Cont}{\ensuremath{\operatorname{Cont}}}
\newcommand{\HB}{\ensuremath{\operatorname{\mathbf{HB}}}}
\newcommand{\Sup}{\ensuremath{\operatorname{\mathbf{Sup}}}}
\newcommand{\Ran}{\ensuremath{\operatorname{\mathbf{Range}}}}
\newcommand{\Sel}{\ensuremath{\operatorname{\mathbf{Sel}}}}
\newcommand{\Vect}{\ensuremath{\operatorname{Vect_\Q}}}
\newcommand{\BS}{\ensuremath{\mathfrak{B}}}
\newcommand{\PF}{\ensuremath{\mathcal{PF}}}

\begin{document}
\begin{abstract}
We determine the computational complexity of the Hahn-Banach
Extension Theorem. To do so, we investigate some basic connections
between reverse mathematics and computable analysis. In particular,
we use Weak K\"{o}nig's Lemma within the framework of computable
analysis to classify incomputable functions of low complexity. By
defining the multi-valued function $\Sep$ and a natural notion of
reducibility for multi-valued functions, we obtain a computational
counterpart of the subsystem of second order arithmetic $\WKL$. We
study analogies and differences between $\WKL$ and the class of
\Sep-computable multi-valued functions. Extending work of Brattka,
we show that a natural multi-valued function associated with the
Hahn-Banach Extension Theorem is \Sep-complete.
\end{abstract}
\thanks{We would like to thank Vasco Brattka for profitable discussion on the
subject.}
\thanks{An extended abstract of this paper appears in the Proceedings
of the Fifth International Conference on Computability and
Complexity in Analysis (CCA 2008), \emph{Electronic Notes in
Theoretical Computer Science} 221, 85--102, 2008,
doi:10.1016/j.entcs.2008.12.009}

\maketitle
%\tableofcontents

\section{Introduction}

In this paper we tackle a problem in computable analysis
(\cite{Weih00} is the main reference in the area) borrowing ideas
and proof techniques from the research program of reverse
mathematics (\cite{sosoa} is the standard reference). The two
subjects share the goal of classifying complexity of mathematical
practice. Reverse mathematics was started by Harvey Friedman in the
1970's (\cite{Fried74}). It adopts a proof-theoretic viewpoint
(although techniques from computability theory are increasingly
important in the subject) and investigates which axioms are needed
to prove a given theorem (see Section \ref{section:rm} for details).
On the other hand, computable analysis extends to computable
separable metric spaces the notions of computability and
incomputability by combining concepts of approximation and of
computation. To this end the representation approach (Type-2 Theory
of Effectivity, TTE), introduced for real functions by Grzegorczyk
and Lacombe (\cite{Grz,Lac}), is used. This approach provides a
realistic and flexible model of computation.

One of the goals of computable analysis is to study and compare
degrees of incomputability of (possibly multi-valued) functions
between separable metric spaces. Multi-valued functions are the
appropriate way of dealing with situations where problems have
non-unique solutions, and have been studied in computable analysis
since \cite{Weih00}. In this paper we introduce a notion of
computable reducibility for multi-valued functions which generalizes
at once both notions of reducibility for single-valued functions
extensively studied by Brattka in \cite{Bra05}. Let $f: \sbsq X
\toto Y$\footnote{The notation $f: \sbsq X \toto Y$ means that $f$
is a multi-valued function with $\dom(f) \subseteq X$ and $\ran(f)
\subseteq Y$. Following \cite[\S1.4]{Weih00}, we view a partial
multi-valued function $f: \sbsq X \toto Y$ as a subset of $X \times
Y$. Then $\dom(f) = \set{x \in X}{\exists y \in Y\, (x,y) \in f}$
and, when $x \in \dom(f)$, we have $f(x) = \set{y \in Y}{(x,y) \in
f}$.} and $g: \sbsq U \toto V$ be two (partial) multi-valued
functions, where $X,Y,U,V$ are separable metric spaces. We say that
$f$ is computably reducible to $g$, and write $f \cred g$, if there
are computable multi-valued functions $h: \sbsq X \toto U$ and $k:
\sbsq X \times V \toto Y$ such that $\emptyset \neq k (x, g(h(x)))
\subseteq f(x)$ (see Definition \ref{def:composition} below for the
definition of composition of multi-valued functions) for all $x \in
\dom(f)$. We use $\creds$ and $\ceq$ to denote the strict order and
the equivalence relation defined in the obvious way.

In \cite{Bra08} Brattka started the study of the separable
Hahn-Banach Theorem from the viewpoint of computable analysis. Given
a computable separable Banach space $X$ consider the multi-valued
function $H_X$ mapping a closed linear subspace $A$ of $X$ and a
bounded linear functional $f: A \to \R$ with $\|f\| = 1$ to the set
of all bounded linear functionals $g: X \to \R$ which extend $f$ and
are such that $\|g\| = 1$. For many computable separable Banach
spaces $X$, it turns out that $H_X$ is incomputable. Brattka does
not establish precisely the degree of incomputability of these
functions, as he shows, in our notation, that $H_X \creds C_1$ for
every $X$. Here $C_1$ is a standard function considered in
computable analysis, the first in a sequence of increasingly
incomputable functions (see Definition \ref{ck} below).

We generalize Brattka's approach and consider the following \lq\lq
global separable Hahn-Banach multi-valued function\rq\rq\ \HB: \HB\
takes as input a separable Banach space $X$, a closed linear
subspace $A \subseteq X$ and a bounded linear functional $f: A \to
\R$ of norm $1$, and gives as output the bounded linear functionals
$g: X \to \R$ which extend $f$ and are such that $\|g\| = 1$.

Reverse mathematics suggests a plausible representative for the
degree of incomputability of \HB. To see this, recall that reverse
mathematics singled out five subsystems of second order arithmetic:
in order of increasing strength these are \RCA, \WKL, \ACA, \ATR\
and \PCA. Most theorems of ordinary mathematics are either provable
in the weak base system \RCA\ or are equivalent, over \RCA, to one
of the other systems. Computable functions naturally correspond to
\RCA\ and is easy to see that $C_1$ (and indeed any $C_k$ with
$k>0$) corresponds to \ACA\ (the correspondence between a system and
a function will be made precise in Section \ref{section:rm&ca}
below). Brown and Simpson (\cite{BS}) showed that, over the base
theory \RCA, the Hahn-Banach Theorem for separable Banach spaces is
equivalent to \WKL. Thus to define a representative for the
incomputability degree of \HB\ we could look for a function in
computable analysis corresponding to \WKL.

We consider the multi-valued function $\Sep: \sbsq \Bai \times \Bai
\toto \Can$ defined on $\set{(p,q) \in \Bai \times \Bai}{\forall n\,
\forall m\, p(n) \neq q(m)}$ by
\[
\Sep (p,q) = \set{r \in \Can}{\forall n (r(p(n)) = 0 \land r(q(n))
= 1)}.
\]
In other words, the domain of $\Sep$ is the collection of pair of
functions from the natural numbers into themselves (i.e.\ of
elements of Baire space) with disjoint ranges, and, for any such
pair $(p,q)$, $\Sep (p,q)$ is the set of the characteristic
functions of sets of natural numbers (i.e.\ elements of Cantor
space) separating the range of $p$ and the range of $q$. Thus \Sep\
corresponds to a statement (strictly connected with
\SI01-separation) which is well-known to be equivalent to \WKL\ (see
\cite[Lemma IV.4.4]{sosoa}). \Sep\ is not computable and, using our
definition of computable reducibility between multi-valued
functions, we obtain, as expected, $\Sep \creds C_1$. We also show
that $\Sep \ceq \Path$, where \Path\ is the multi-valued function
associating to an infinite subtree of $\Seqd$ the set of its
infinite paths. Moreover we prove that $\Sep \ceq \HB$, establishing
the degree of incomputability of the separable Hahn-Banach Theorem.

The \lq\lq reversal\rq\rq\ in Brown and Simpson's result (i.e.\ the
proof that the separable Hahn-Banach Theorem implies \WKL) is based
on a construction due to Bishop, Metakides, Nerode and Shore
(\cite{MNS}) and appears also in \cite[Theorem IV.9.4]{sosoa}. We
exploit the ideas of this proof to show $\Sep \cred \HB$. The
original proof by Brown and Simpson of the \lq\lq forward
direction\rq\rq\ (showing that \WKL\ proves the separable
Hahn-Banach Theorem) has been simplified first by Shioji and Tanaka
(\cite{ST}, this is essentially the proof contained in \cite[\S
IV.9]{sosoa}) and then by Humphreys and Simpson (\cite{HS}). No
details of these or other proofs of the Hahn-Banach Theorem are
needed for showing $\HB \cred \Sep$. Brattka noticed the possibility
of avoiding these details in \cite{Bra08} and wrote:
\emph{Surprisingly, the proof of this theorem does not require a
constructivization of the classical proof but just an \lq\lq
external analysis\rq\rq}. We explain this fact by observing that the
computable analyst is allowed to conduct an unbounded search for an
object that is guaranteed to exist by (nonconstructive) mathematical
knowledge, whereas the reverse mathematician has the burden of an
existence proof with limited means. We give another instance of this
phenomenon in Example \ref{compact} below.

Of course, each of the mathematical objects mentioned above needs
some \lq\lq coding\rq\rq\ (in reverse mathematics jargon) or \lq\lq
representation\rq\rq\ (using computable analysis terminology). In
this respect the computable analysis and the reverse mathematics
traditions have developed slightly different approaches to separable
Banach spaces.\smallskip

The plan of the paper is as follows. Sections \ref{section:ca} and
\ref{section:rm} are brief introductions to computable analysis and
reverse mathematics, respectively. The reader with some basic
knowledge in one of these fields can safely skip the corresponding
section and refer back to it when needed. Section
\ref{section:multi} deals with multi-valued functions and computable
reductions among them. In Section \ref{section:rm&ca} we compare
reverse mathematics and computable analysis. We show the
similarities of the two approaches, but also note that results
cannot be translated automatically in either direction. The
multi-valued function \Sep\ is studied in Section \ref{section:Sep}.
Section \ref{section:BS} sets up the study of Banach spaces in
computable analysis, while Section \ref{section:HB} contains the
proof of $\HB \ceq \Sep$.\bigskip

\subsection{Notation for sequences}
We finish this introduction by establishing our notation for finite
and infinite sequences of natural numbers.

Let \Seq, resp.\ \Bai, be the sets of all finite, resp.\ infinite,
sequences of natural numbers. When $s \in \Seq$ we use $|s|$ to
denote its length and, for $i<|s|$, $s(i)$ to denote the $(i+1)$-th
element in the sequence. Similarly, $p(i)$ is defined for every $i$
when $p \in \Bai$. Let $\N^n$ be set of all $s \in \Seq$ with
$|s|=n$. We use $\l$ to denote the empty sequence, i.e.\ the only
element of $\N^0$, and $\bar{0}$ to denote the infinite sequence
which always takes value $0$. When $s,t \in \Seq$ we write $s
\sqsubseteq t$ to mean that $s$ is an initial segment of $t$. $s
\conc t$ is the sequence obtained by concatenating $t$ after $s$,
and when $k \in \N$, $s*k$ abbreviates $s \conc \la k \ra$, and
$k*s$ abbreviates $\la k \ra \conc s$. When $p \in \Bai$ we write
also $s \conc p$ and $k*p$, which are the obvious elements of
$\Bai$. If $p \in \Bai$ and $n \in \N$ we write $p[n]$ for the
sequence $\la p(0), p(1), \dots, p(n-1) \ra \in \N^n$. If $p,q \in
\Bai$ we let $p \oplus q \in \Bai$ be such that $(p \oplus q) (2i) =
p(i)$ and $(p \oplus q) (2i+1) = q(i)$.

We define \Can, \Seqd, and $2^n$ as the subsets of \Bai, \Seq, and
$\N^n$ whose elements take values in $\{0,1\}$.

We fix a bijection between \Seq\ and $\N$ and, as usual in the
literature, we identify an element of \Seq\ with the corresponding
natural number. We assume that the maps $s \mapsto |s|$, $(s,i)
\mapsto s(i)$, $k \mapsto \la k \ra$, and $(s,t) \mapsto s \conc t$
are all computable.

Of course, \Bai\ has a natural topology, namely the product topology
starting from the discrete topology on $\N$. When we view \Bai\ as a
topological space we call it the \emph{Baire space}. Similarly,
\Can\ with the relative topology is the \emph{Cantor space}.

\section{Computable analysis}\label{section:ca}

\subsection{TTE computability}
In contrast with the case of natural numbers, several nonequivalent
approaches to computability theory for the reals have been proposed
in the literature. We work in the framework of the so called Type-2
Theory of Effectivity (TTE), which finds a systematic foundation in
\cite{Weih00}. TTE extends the ordinary notion of Turing
computability to second countable $T_0$-topological spaces, and
therefore deals with computability over the reals as a particular
case within a more general theory.

The basic idea of TTE is that concrete computing machines do not
manipulate directly abstract mathematical objects, but they perform
computations on sequences of digits which are codings for such
objects. In general, mathematical objects require an infinite amount
of information to be completely described, and it is therefore
natural to extend the ordinary theory of computation to infinite
sequences. This extension does not compromise the concreteness of
the model, since computations on infinite sequences have a very
natural translation in terms of ordinary Turing computations on
finite sequences (see \cite[Lemma 2.1.11]{Weih00}). The most
important feature that differentiates TTE Turing machines from
ordinary Turing machines is the fact that they are conceived to work
on infinite strings of $0$'s and $1$'s, and they do that according
to the following specifications. TTE Turing machines have one input
tape, one working tape and one output tape. Each tape is equipped
with a head. All ordinary instructions for Turing machines are
allowed for the working tape, while the head of the input tape can
only read and move rightward, and the head of the output tape can
only write and move rightward. These limitations (in particular,
those for the output tape) imply the impossibility of correcting the
output; once a digit is written, it cannot be canceled or changed.
Hence at each stage of the computation the partial output is
reliable (this is the most we can ask, since in finitely many steps
we never obtain a complete output).

It is straightforward to enumerate all TTE Turing machines and let
$M_k$ be the $k$-th such machine. Let $\xi_k: \sbsq \Bai \to \Bai$
be the partial function computed by $M_k$ as follows. Given $p \in
\Bai$ let $p'$ consist of $p(0)$ $1$'s followed by a single $0$,
$p(1)$ $1$'s, and so on; write $p'$ on the input tape and start
$M_k$; if the computation is infinite and the output tape eventually
contains an infinite sequence of $0$'s and $1$'s with infinitely
many $0$'s we translate back to an element of \Bai\ which is $\xi_k
(p)$; otherwise $p \notin \dom(\xi_k)$.

Notice that $\dom(\xi_k)$ is a $G_\d$ subset of \Bai\ for every $k$.

\begin{definition}[Computable functions on \Bai]
We say that a function $F: \sbsq \Bai \to \Bai$ is \emph{computable}
if there exists $k$ such that $\dom(F) \subseteq \dom(\xi_k)$ and
$F(p) = \xi_k(p)$ for every $p \in \dom(F)$.
\end{definition}

As noticed by Weihrauch (\cite[p. 38]{Weih00}), TTE Turing machines
can be viewed as ordinary oracle Turing machines; the oracle
supplies the information about the input and the $n$-th bit of the
output is computed when we give $n$ as input to the oracle Turing
machine. Therefore the computable partial functions from \Bai\ to
\Bai\ coincide with the computable (or recursive)
functionals\footnote{Beware that in some literature \lq\lq
functional\rq\rq\ means function from \Bai\ to $\N$, rather than
function from \Bai\ to \Bai\ as here.} of classical computability
(or recursion) theory, also known as Lachlan functionals.

The restrictions on the instructions allowed in TTE Turing machines
imply the following fact (\cite[Theorem 2.2.3]{Weih00}).

\begin{lemma}\label{comp->cont}
Every computable function $F:\sbsq \Bai \to \Bai$ is continuous.
\end{lemma}

We transfer the notion of computability for the Baire space to
spaces with cardinality less than or equal to the continuum using
the notion of representation.

\begin{definition}[Representations and represented spaces]\label{def:representation}
A \emph{representation} $\s_X$ of a set $X$ is a surjective function
$\s_X: \sbsq \Bai \to X$. The pair $(X,\s_X)$ is a \emph{represented
space}.

If $x \in X$ a \emph{$\s_X$-name} for $x$ is any $p \in \Bai$ such
that $\s_X(p) = x$.

We say that $x$ is \emph{$\s_X$-computable} when it has a computable
$\s_X$-name $p$ (i.e.\ $graph(p)$ is a computable set).
\end{definition}

\subsection{Effective metric spaces}
The definition of representation is too general for practical
purposes, as it allows an object in $X$ to be coded by arbitrary
sequences. However, there are important cases in which we can find
meaningful representations, for example when $X$ is a separable
metric space.

\begin{definition}[Effective metric space]\label{def:ems}
An \emph{effective metric space} is a triple $(X,d,a)$ where
\begin{itemize}
    \item $(X,d)$ is a separable metric space;
    \item $a: \N \to X$ is a dense sequence in $X$.
\end{itemize}
If there is no danger of confusion, we often write $X$ in place of
$(X,d,a)$.

We equip every effective metric space $(X,d,a)$ with the
\emph{Cauchy representation} $\d_X: \sbsq \Bai \to X$, such that $p
\in \dom(\d_X)$ if and only if for all $i$ and all $j \geq i$,
$d(a(p(i)), a(p(j))) \leq 2^{-i}$, and $\d_X(p) = x$ if and only if
$\lim a(p(n)) = x$. In other words, $p \in \Bai$ is a name for $x$
when $p$ encodes a Cauchy sequence of elements in the fixed dense
subset of $X$ which \emph{converges effectively} to $x$.

A \emph{rational open ball} in $(X,d,a)$ is an open ball of the form
$B^X (c;\a) = \set{x \in X}{d(x,c) < \a}$ with $c \in \ran(a)$, and
$\a \in \Q^+\cup \{0\}$.
\end{definition}

In particular, we have the effective metric space $(\R, d, a_\Q)$,
where $d(x,y) = |x-y|$ and $a_\Q$ is a standard computable
enumeration of the set of the rational numbers (it is convenient to
assume $a_\Q(0)=0$ and $a_\Q(1)=1$).

The notion of effective metric space can be generalized.

\begin{definition}[Effective topological space]\label{def:ets}
An \emph{effective topological space} is a triple $(X, \tau, u)$,
where $\tau$ is a second countable $T_0$-topology on $X$ and $u: \N
\to \mathcal{P} (X)$ is an enumeration of a sub-base of $\tau$.

Each effective topological space $(X, \tau, u)$ has a \emph{standard
representation} $\d_X$ such that $\d_X(p) = x \in X$ if and only if
$\ran(p) = \set{n}{x \in u(n)}$.
\end{definition}

It is immediate that effective metric spaces are particular examples
of effective $T_0$-topological spaces. In fact, if $(X,d,a)$ is an
effective metric space we let $u$ enumerate the rational open balls
of $X$. We will always assume that there exist computable functions
$c$ and $r$ such that $u(n)$ has center $a (c(n))$ and radius $a_\Q
(r(n))$. In this context we usually write $B^X_n$ in place of
$u(n)$.

The Cauchy representation of an effective metric space $X$ is
equivalent to the representation of $X$ considered as an effective
topological space. This equivalence means that each representation
is reducible to the other, where a representation $\d$ of a set $X$
is \emph{reducible} to a representation $\s$ of the same set when
there is a continuous function $F: \sbsq \Bai \to \Bai$ such that
$\d(p) = \s(F(p))$ for all $p \in \dom(\d)$. A representation of $X$
which is equivalent to the standard representation is said to be
\emph{admissible} for $X$.

\begin{definition}[Realizers]\label{realizer}
Given represented spaces $(X, \s_X)$ and $(Y, \s_Y)$ and a partial
function $f: \sbsq X \to Y$, we say that $F:\sbsq \Bai \to \Bai$ is
a \emph{$(\s_X, \s_Y)$-realizer} of $f$ when $f(\s_X(p)) =
\s_Y(F(p))$, for all $p \in \dom(f \circ \s_X)$.

The function $f$ is said to be $(\s_X, \s_Y)$-\emph{computable} if
it has a computable $(\s_X, \s_Y)$-realizer. In practice we often
omit explicit mention of the representations and write just
computable.
\end{definition}

Using the notion of realizer we thus extend the notion of computable
from the Baire space to the effective topological spaces. This
extension is particularly successful when we use admissible
representations, as the following results (due to Kreitz and
Weihrauch) show.

\begin{theorem}\label{main}
Let $X$ and $Y$ be effective topological spaces with admissible
representations $\s_X$ and $\s_Y$. A function $f: \sbsq X \to Y$ is
continuous if and only if it has a continuous $(\s_X,
\s_Y)$-realizer.
\end{theorem}

\begin{cor}\label{comp->cont2}
Let $X$ and $Y$ be effective topological spaces with admissible
representations $\s_X$ and $\s_Y$. Then every function $f: \sbsq X
\to Y$ which is $(\s_X, \s_Y)$-computable is continuous.
\end{cor}

Corollary \ref{comp->cont2} is an extension of Lemma
\ref{comp->cont}. We point out that Theorem \ref{main} and Corollary
\ref{comp->cont2} hold in particular for effective metric spaces and
Cauchy representations.

The notions of effective metric and effective topological spaces in
their complete generality have no computational content. In fact,
notwithstanding the established terminology (\cite{Weih00}), we are
not requiring any \lq\lq effectivity\rq\rq\  property (even the
computable enumeration of the rational open balls of an effective
metric space is nothing but an enumeration of pairs of natural
numbers). In the case of effective metric spaces, the natural \lq\lq
effective\rq\rq\ requirement is the computability of the distance
between points.

\begin{definition}[Computable metric space]\label{def:cms}
A \emph{computable metric space} is an effective metric space
$(X,d,a)$ such that the function $(n,m) \mapsto d(a(n), a(m))$ is
computable.
\end{definition}

If $X$ is a computable metric space it is straightforward that the
distance function is $((\d_X, \d_X), \d_\R)$-computable. Typical
examples of computable metric spaces are $\R$ and the Baire space
(recall that for $p, q \in \Bai$ such that $p \neq q$ we let $d(p,q)
= 2^{-i}$ for the least $i$ such that $p(i) \neq q(i)$).

In the case of effective $T_0$-topological spaces, the \lq\lq
effective\rq\rq\ requirement is the computability of the operation
of intersection of open sets (see \cite{GW}).

\subsection{Representations of continuous functions}\label{subsect:rcf}
Notice that the set of all continuous partial functions on the Baire
space is too large to have a representation. However, every partial
continuous function $f: \sbsq \Bai \to \Bai$ has a continuous
extension to a $G_\d$ set (\cite[Theorem 2.3.8]{Weih00}; this is an
instance of a classical result due to Kuratowski, see e.g.\
\cite[Theorem 3.8]{cdst}). Thus it suffices to represent
\[
\Cont = \set{F:\sbsq \Bai \to \Bai}{F \text{ is continuous and
$\dom(F)$ is $G_\d$}}.
\]
Lemma \ref{comp->cont} implies that each computable $F: \sbsq \Bai
\to \Bai$ has an extension in \Cont. Define $\eta: \Bai \to \Cont$
by $\eta (k*p) (q) = \xi_k (p \oplus q)$, for $k \in \N$ and $p,q
\in \Bai$. $\eta$ is a representation of \Cont.

Given effective metric spaces $X$ and $Y$, we define a
representation $[\d_X \to \d_Y]$ of the set $\mathcal{C} (X,Y)$ of
total continuous functions from $X$ into $Y$ by $[\d_X \to \d_Y] (p)
= f$ if and only if $\e(p)$ is a $(\d_X, \d_Y)$-realizer of $f$.
This representation satisfies the following fundamental properties:
\begin{description}
    \item[Evaluation] the map $(f,x) \mapsto f(x)$ is $(([\d_X
        \to \d_Y], \d_X), \d_Y)$-computable;
    \item[Type conversion] let $(Z,\s_Z)$ be a represented
        space; every function $g: Z \times X \to Y$ is $((\s_Z,
        \d_X), \d_Y)$-computable if and only if $\hat{g}: Z \to
        \mathcal{C} (X,Y)$, defined by $\hat{g} (z) (x) =
        g(z,x)$, is $(\s_Z, [\d_X \to \d_Y])$-computable.
\end{description}
The evaluation and type conversion properties witness the
reliability of the simulation of continuous functions on separable
metric spaces via realizers.

\subsection{Borel complexity}
Computable analysis provides a method to classify incomputable
functions between separable metric spaces in complexity hierarchies,
analogously to the classification of functions from $\N$ to $\N$
pursued in classical computability theory. In particular,
\cite{Bra05} studied the following functions of strictly increasing
complexity.

\begin{definition}[The $C_k$'s]\label{ck}
For every $k \in \N$ let $C_k:\Bai \to \Bai$ be defined by
\[
C_k(p)(n)= \begin{cases}
0 & \text{if $\exists n_k\, \forall n_{k-1}\, \exists n_{k-2}\, \dots Qn_1\,
p(\la n, n_k, n_{k-1}, \dots, n_1 \ra) \neq 0$;}\\
1 & \text{otherwise.}\\
\end{cases}
\]
where $Q$ is $\exists$ when $k$ is odd and $\forall$ when $k$ is
even.
\end{definition}

Using natural representations for Borel sets of each given finite
level, Brattka (\cite{Bra05}) says that a function $f: \sbsq X \to
Y$, for $X$ and $Y$ computable metric spaces, is
\emph{\SI0k-computable} (for $k\geq 1$) if there exists a computable
function that maps every name of an open set $U \subseteq Y$ to a
name of a \SI0k set $V \subseteq X$ such that $f^{-1}(U) = V \cap
\dom(f)$. It follows immediately that every \SI0k-computable
function is \SI0k-measurable (equivalently, of Baire class $k-1$).
Brattka shows that $f$ is \SI0{k+1}-computable if and only if $f$ is
computably reducible to $C_k$\footnote{We refer the reader to
Section \ref{section:multi} below for the definition of
reducibility.}.

\section{Reverse mathematics}\label{section:rm}

In the 1970's Harvey Friedman started the research program of
reverse mathematics, which was pursued in the two next decades by
Steve Simpson and his students and increasingly by other
researchers. Nowadays reverse mathematics is an important area of
mathematical logic, crossing the boundary between computability
theory and proof theory, but employing ideas and techniques also
from model theory and set theory. We refer the reader to
\cite{sosoa} for details about the topics we will sketch in this
section (the collection \cite{RM2001} documents more recent
advances).

Reverse mathematics searches in a systematic way for equivalences
between different statements with respect to some base theory (which
does not prove any of them) in the context of subsystems of second
order arithmetic. Recall that the language \LL\ of second order
arithmetic has variables for natural numbers and variables for sets
of natural numbers, constant symbols $0$ and $1$, binary function
symbols for addition and product of natural numbers, symbols for
equality and the order relation on the natural numbers and for
membership between a natural number and a set. Second order
arithmetic is the \LL-theory with classical logic consisting of the
axioms stating that $(\N,0,1,{+},{{\cdot}},{<})$ is a commutative ordered
semiring with identity, the induction scheme for arbitrary formulas,
and the comprehension scheme for sets of natural numbers defined by
arbitrary formulas. Hermann Weyl and Hilbert and Bernays already
noticed that \LL\ was rich enough to express, using appropriate
codings, significant parts of mathematical practice, and that many
mathematical theorems were provable in (fragments of) second order
arithmetic.

Formulas of \LL\ are classified in the usual hierarchies: those with
no set quantifiers and only bounded number quantifiers are \DE00,
while counting the number of alternating unbounded number
quantifiers we obtain the classification of all arithmetical (=
without set quantifiers) formulas as \SI0n and \PI0n formulas (one
uses \SI{}{} or \PI{}{} depending on the type of the first
quantifier in the formula, existential in the former, universal in
the latter). Formulas with set quantifiers in front of an
arithmetical formula are classified by counting their alternations
as \SI1n and \PI1n. A formula is \DE in a certain theory if it is
equivalent in that theory both to a \SI in formula and to a \PI in
formula.

Reverse mathematics starts with the fairly weak base theory \RCA,
where the induction scheme and the comprehension scheme are
restricted respectively to \SI01 and \DE01 formulas. \RCA\ is strong
enough to prove some basic results about many mathematical
structures, but too weak for many others.

If a theorem $T$ is expressible in \LL\ but unprovable in \RCA,
reverse mathematics asks the question: what is the weakest axiom we
can add to \RCA\ to obtain a theory that proves $T$? In principle,
we could expect that this question has a different answer for each
$T$. The \lq\lq discovery\rq\rq\ of reverse mathematics is that this
is not the case. In fact, most theorems of ordinary mathematics
expressible in \LL\ are either provable in \RCA\ or equivalent over
\RCA\ to one of the following four subsystems of second order
arithmetic, listed in order of increasing strength: \WKL, \ACA,
\ATR, and \PCA. This leads to a neat picture where theorems
belonging to quite different areas of mathematics are classified in
five levels, roughly corresponding to the mathematical principles
used in their proofs. \RCA\ corresponds to \lq\lq computable
mathematics\rq\rq, \WKL\ embodies a compactness principle, \ACA\ is
linked to sequential compactness, \ATR\ allows for transfinite
arguments, \PCA\ includes impredicative principles.

In this paper we will refer extensively to \WKL\ and, in passing, to
\ACA. Therefore we describe these two theories in a little more
detail.

\ACA\ is obtained from \RCA\ by extending the comprehension scheme
to all arithmetical formulas. The statements without set variables
provable in \ACA\ coincide exactly with the theorems of Peano
arithmetic, so that in particular the consistency strength of the
two theories is the same. Within \ACA\ one can develop a fairly
extensive theory of continuous functions, using the completeness of
the real line as an important tool. \ACA\ proves (and often turns
out to be equivalent to) also many basic theorems about countable
fields, rings, and vector spaces.

To obtain \WKL\ we add to \RCA\ the statement of Weak K\"{o}nig's Lemma,
i.e.\ every infinite binary tree has a path, which is essentially
the compactness of Cantor space. An equivalent statement, clearly
showing that \WKL\ is stronger than \RCA\ and weaker than \ACA, is
\SI01-separation: if $\varphi(n)$ and $\psi(n)$ are \SI01-formulas
such that $\forall n\, \neg (\varphi(n) \land \psi(n))$ there exists
a set $X$ such that $\varphi(n) \implies n \in X$ and  $\psi(n)
\implies n \notin X$ for all $n$. \WKL\ and \RCA\ have the same
consistency strength of Primitive Recursive Arithmetic, and are thus
proof-theoretically fairly weak. Nevertheless, \WKL\ proves (and
often turns out to be equivalent to) a substantial amount of
classical mathematical theorems, including many results about
real-valued functions, basic Banach space facts, etc. For example,
\WKL\ is equivalent, over \RCA, to the Peano-Cauchy existence
theorem for solutions of ordinary differential equations.

\section{Multi-valued functions in computable analysis}\label{section:multi}

The main goal of this section is to give the definition of
reducibility of multi-valued functions.

Since we will often compose multi-valued functions, we spell out
Weihrauch's definition for this operation.

\begin{definition}[Composition of multi-valued functions]\label{def:composition}
Given two (partial) multi-valued functions $f: \sbsq X \toto Y$ and
$g: \sbsq Y \toto Z$, the \emph{composition} $g \circ f: \sbsq X
\toto Z$ is the multi-valued function defined by
\begin{itemize}
    \item $\dom (g \circ f) = \set{x \in \dom(f)}{f(x) \subseteq
        \dom(g)};$
    \item $\forall x \in \dom(g \circ f)\, (g \circ f) (x) =
        \bigcup_{y \in f(x)} g(y).$
\end{itemize}
\end{definition}

To define the notion of computable multi-valued function we look at
realizers.

\begin{definition}[Realizers of multi-valued functions]\label{def:realizer}
Let $(X,\s_X)$ and $(Y,\s_Y)$ be represented spaces and $f: \sbsq X
\toto Y$. A \emph{$(\s_X,\s_Y)$-realizer} for $f$ is a
(single-valued) function $F: \sbsq \Bai \to \Bai$ such that
$\s_Y(F(p)) \in f(\s_X(p))$ for every $p \in \dom(f \circ \s_X)$.
\end{definition}

Notice that in Definition \ref{def:realizer} we do not require that
$\s_X (p) = \s_X (p')$ implies $\s_Y (F(p)) = \s_Y (F(p'))$. In
other words a realizer does not, in general, lift to a single-valued
selector for the multi-valued function.

\begin{definition}[Computability of multi-valued
functions]\label{def:compmvfunct} Let $(X, \s_X)$ and $(Y, \s_Y)$ be
represented spaces. A multi-valued function $f: \sbsq X \toto Y$ is
\emph{$(\s_X, \s_Y)$-computable} if it has a computable $(\s_X,
\s_Y)$-realizer. In practice we often omit explicit mention of the
representations and write just computable.
\end{definition}

Our definition of computable multi-valued function agrees with
\cite[Definition 3.1.3.4]{Weih00} and \cite[p.21]{Bra05}. Notice
however that Brattka's paper includes also the definition of
\SI01-computable multi-valued function; for single-valued functions
the notions of computable and \SI01-computable coincide, but for
arbitrary multi-valued functions the latter is stronger.

\subsection{Reducibility of multi-valued functions}
We now define the notion of computable reducibility for multi-valued
functions. The intuitive idea is that one problem is reducible to
another, provided that whenever we have a method to compute a
solution for the second problem, we can uniformly find a way to
compute a solution for the first one. This generalizes the notion of
reducibility between single-valued functions investigated in
\cite{Bra05} and extensively used in recent work in computable
analysis. Actually, in \cite{Bra05} there are two distinct notions,
introduced in Definitions 5.1 and 7.1, of computable reducibility
between single-valued functions. Our definition generalizes the
former, and Lemma \ref{equivred} below shows that the generalization
of the latter (realizer reducibility) leads to an equivalent
concept\footnote{Brattka's notion of realizer reducibility, as well
its generalization to the case of multi-valued functions (Lemma
\ref{equivred}.(ii)), are particular cases of Wadge's reducibility
for sets of functions as defined in \cite[Def. 8.2.5]{Weih00}.}.
Thus the notion of computable reducibility appears to be more robust
in the multi-valued setting.

\begin{definition}[Reducibility of multi-valued functions]\label{def:redmulti}
Let $(X,\s_X)$, $(Y,\s_Y)$, $(Z,\s_Z)$, $(W,\s_W)$ be represented
spaces. Let $f: \sbsq X \toto Y$ and $g: \sbsq Z \toto W$ be
multi-valued functions. We say that $f$ is \emph{computably
reducible} to $g$, and write $f \cred g$, if there exist computable
multi-valued functions $h: \sbsq X \toto Z$ and $k: \sbsq X \times W
\toto Y$ such that $k(x, (g \circ h)(x)) \subseteq f(x)$ for all $x
\in \dom(f)$.
\end{definition}

Notice that when $f$ and $g$ are single-valued $k$ is single-valued
on $\set{(x,(g \circ h)(x))}{x \in \dom(f)}$, but it may be the case
that $h$ is not single-valued. Therefore the restriction of our
notion of computable reducibility to single-valued functions is
weaker than Brattka's notion of computable reducibility for
single-valued functions. However when dealing with multi-valued
functions it is natural to allow $h$ and $k$ to be multi-valued as
well. As we have pointed out, the following Lemma gives further
support to our definition, by showing that it coincides with the
natural generalization of Brattka's notion of realizer reducibility.

\begin{lemma}\label{equivred}
Let $(X,\s_X)$, $(Y,\s_Y)$, $(Z,\s_Z)$, $(W,\s_W)$ be represented
spaces. Let $f: \sbsq X \toto Y$ and $g: \sbsq Z \toto W$ be
multi-valued functions. The following are equivalent
\begin{enumerate}[(i)]
  \item $f \cred g$;
  \item there exist computable functions $H: \sbsq \Bai \to
      \Bai$ and $K: \sbsq \Bai \times \Bai \to \Bai$ such that
      $p \mapsto K(p, (G \circ H)(p))$ is a realizer for $f$
      whenever $G$ is a realizer for $g$.
\end{enumerate}\end{lemma}
\begin{proof}
First assume $f \cred g$ and let the computable multi-valued
functions $h$ and $k$ witness this. Let $H$ and $K$, respectively,
be computable realizers for $h$ and $k$. Suppose $G$ is a realizer
for $g$; we claim that $p \mapsto K(p,(G \circ H)(p))$ is a realizer
for $f$. In fact if $p \in \dom (f \circ \s_X)$ then $(\s_Z \circ H)
(p) \in (h \circ \s_X) (p)$ and hence $(\s_W \circ G \circ H) (p)
\in (g \circ h \circ \s_X) (p)$, so that $(\s_Y \circ K) (p, (G
\circ H)(p)) \in k(\s_X(p), (g \circ h \circ \s_X) (p)) \subseteq (f
\circ \s_X) (p)$.

Now suppose (ii) holds and let $H$ and $K$ witness this. Define $h$
and $k$ by
\begin{gather*}
    h(x) = \set{(\s_Z \circ H)(p)}{\s_X (p) = x} \quad \text{and}\\
    k(x,w) = \set{(\s_Y \circ K) (p,p')}{\s_X (p) = x \land \s_W (p') = w}.
\end{gather*}
Since $H$ and $K$ are computable realizers for $h$ and $k$
respectively, the latter are computable multi-valued functions.

To check that $h$ and $k$ witness $f \cred g$ let $x \in \dom(f)$
and suppose $y \in k(x, (g \circ h)(x))$. There exist $z \in h(x)$
and $w \in g(z)$ such that $y \in k(x,w)$. By definition of $k$ let
$p,p'$ be such that $\s_X(p) = x$, $\s_W(p') = w$ and $y = (\s_Y
\circ K) (p,p')$. Let $G$ be a realizer for $g$ such that $p' = (G
\circ H) (p)$. Then
\[
y = (\s_Y \circ K) (p,p') = (\s_Y \circ K) (p, (G \circ H) (p)) \in (f \circ \s_X) (p) = f(x),
\]
where membership follows from the fact that $p \mapsto K(p, (G \circ
H)(p))$ is a realizer for $f$. We have thus shown $k(x, (g \circ
h)(x)) \subseteq f(x)$, as needed.
\end{proof}

Since transitivity of $\cred$ for multi-valued functions is not
immediately obvious, we state it explicitly.

\begin{lemma}\label{trans-cred}
$\cred$ is transitive.
\end{lemma}
\begin{proof}
Let $f: \sbsq X \toto Y$, $g: \sbsq Z \toto W$ and $\ell: \sbsq U
\toto V$ be multi-valued functions. Let $h$ and $k$ witness $f \cred
g$, while $h'$ and $k'$ witness $g \cred \ell$. It is easy to check
that $h' \circ h$ and the map $(x,v) \mapsto k(x,k'(h(x),v))$
witness $f \cred \ell$.
\end{proof}

Thus $\cred$ is a preorder (reflexivity is obvious) and we can give
the usual definitions.

\begin{definition}
As usual we use $\creds$ and $\ceq$ for the strict relation and the
equivalence relation arising from $\cred$.
\end{definition}

We now prove two simple Lemmas about $\cred$.

\begin{lemma}\label{superset}
Let $f,g: \sbsq X \toto Y$ be multi-valued functions such that
$\dom(f) \subseteq \dom(g)$ and $g(x) \subseteq f(x)$ for every $x
\in \dom(f)$. Then $f \cred g$.
\end{lemma}
\begin{proof}
It is straightforward to check that the identity on $X$ and
projection on the second coordinate from $X \times Y$ witness this.
\end{proof}

\begin{lemma}\label{compositionwithcomputable}
Let $h: \sbsq X \toto Y$ and $g: \sbsq Z \toto W$ be computable
multi-valued functions. For any multi-valued function $f: \sbsq Y
\toto Z$ we have $(g \circ f \circ h) \cred f$.
\end{lemma}
\begin{proof}
It is straightforward to check that $h$ and $(x,z) \mapsto g(z)$
witness this.
\end{proof}

Our definition of \SI0k-computability for multi-valued functions is
motivated by the characterization of this notion for singled-valued
functions of Theorems 5.5 and 7.6 (one for each notion of
reducibility) in \cite{Bra05}. The reader should however be aware
that Brattka defined a notion of \SI0k-computability for
multi-valued functions which is properly stronger than ours
(\cite[Definition 3.5]{Bra05}).

\begin{definition}[\SI0k-computable and \SI0k-complete]\label{kcompcompl}
Let $k \geq 1$ and $(X,\s_X)$, $(Y,\s_Y)$ be represented spaces. A
multi-valued function $f: \sbsq X \toto Y$ is \emph{\SI0k-computable}
if $f \cred C_{k-1}$, and \emph{\SI0k-complete} if $f \ceq C_{k-1}$.
\end{definition}

Lemma \ref{equivred} above and Theorem 7.6 in \cite{Bra05} imply
that a multi-valued function is \SI0k-computable if and only if it
has a \SI0k-computable realizer.

\section{Reverse mathematics and computable analysis}\label{section:rm&ca}

\subsection{Correspondence between statements of second order arithmetic and functions}
Many mathematical statements expressed in \LL\ have the form
\[
\forall X (\psi(X) \implies \exists Y\, \varphi (X,Y))
\]
where $X$ and $Y$ range over sets of natural numbers. Here are a few
examples (we use the standard coding techniques for expressing in
\RCA\ functions, real numbers, sequences, etc.).
\begin{enumerate}[(1)]
  \item\label{ex:path} The statement of Weak K\"{o}nig's Lemma (the
      main axiom of \WKL) is
\[
\forall T (T \text{ is an infinite binary tree} \implies \exists p\,
p \text{ is an infinite path in } T).
\]
  \item\label{ex:range} The existence of the range of any
      function is
\[
\forall p (p: \N \to \N \implies \exists Y\,
\forall m (m \in Y \iff \exists n\, m = p(n))).
\]
  \item\label{ex:sup} The existence of the least upper bound for
      any sequence in $I = [0,1]$ is
\[
\forall \la x_n : n \in \N\ra (\forall n\, x_n \in I \implies
\exists x (\forall n\, x_n \leq x \land \forall k\, \exists n\,
x < x_n + 2^{-k})).
\]
  \item Separation of disjoint ranges is
\[
\forall p,q (p,q: \N \to \N \land \forall n,m\, p(n) \neq q(m)
\implies \exists Y\, \forall n (p(n) \in Y \land q(n) \notin Y)).
\]
  \item\label{ex:HBc} The statement of the Heine-Borel
      compactness of the interval $I$ is
\begin{multline*}
\forall \la I_k : k \in \N\ra (\forall k\, I_k \subseteq I
\text{ is an interval with rational endpoints} \land\\
\land I = \bigcup_{k \in \N} I_k \implies \exists n\, I =
\bigcup_{k<n} I_k).
\end{multline*}
  \item The statement of the separable Hahn-Banach Theorem is
\begin{multline*}
\forall X, A, f (X \text{ is a separable Banach space} \land A
\text{ is a closed linear subspace of } X \land\\
\land f \text{ is a bounded linear functional on } A \implies\\
\implies \exists g (g \text{ is a bounded linear functional on
$X$ extending } f \land \|g\| = \|f\|)).
\end{multline*}
\end{enumerate}

If $\forall X (\psi(X) \implies \exists! Y\, \varphi (X,Y))$ holds
(this is the case in (\ref{ex:range}) and (\ref{ex:sup}) above) it
is natural to consider the partial function $f: \sbsq \mathcal{P}
(\N) \to \mathcal{P} (\N)$ with $\dom (f) = \set{X \in \mathcal{P}
(\N)}{\psi(X)}$ such that $\varphi (X, f(X))$ for every $X \in
\dom(f)$. When the uniqueness condition fails we could consider all
possible functions with the properties above. However it seems more
useful to study the multi-valued function $f: \sbsq \mathcal{P} (\N)
\toto \mathcal{P} (\N)$ defined by $f(X) = \set{Y}{\varphi(X,Y)}$
for all $X$ such that $\psi(X)$.

\begin{remark}
In many cases, including some of the examples given above, it is
best to view the domain and the range of $f$ as represented spaces
different from $\mathcal{P} (\N)$, thus unraveling the coding used
in the reverse mathematics approach. E.g.\ the functions arising
from examples (\ref{ex:path}) and (\ref{ex:sup}) are best viewed
respectively as a partial multi-valued function from $\mathcal{P}
(\Seqd)$ to \Can\ and a total single-valued function from $I^\N$ to
$I$.
\end{remark}

We have thus associated to the mathematical statement expressed in
\LL\ a function between represented spaces which can be studied
within the framework of computable analysis. Notice that the lack of
restrictions on the complexity of $\psi$ corresponds to the
principle of computable analysis stating that \lq\lq the user is
responsible for the correctness of the input\rq\rq\ (see
\cite[\S6]{GSW} for a discussion).

We can also reverse the procedure. If we want to study from the
viewpoint of computable analysis a multi-valued function $f : \sbsq
X \toto Y$, we can look at the reverse mathematics of the statement
\[
\forall x (x \in \dom(f) \implies \exists y \in Y\, y \in f(x)),
\]
with the hope of gaining some useful insight. E.g.\ if $k \geq 1$,
from $C_k$ we obtain the statement
\begin{multline*}
\forall p (p \in \Bai \implies \exists q \in \Can\, \forall n \\
(q(n) = 0 \iff \exists n_k\, \forall n_{k-1}\, \exists n_{k-2}\,
\dots Qn_1\, p(\la n, n_k, n_{k-1}, \dots, n_1 \ra) \neq 0)),
\end{multline*}
which is easily seen to be equivalent (over \RCA) to
\SI0k-comprehension.

In any case, we expect some connection between the proof-theoretic
strength of the statement and the computability strength of the
function.

Notice that statements corresponding to functions belonging to
different degrees of incomputability may collapse into a single
system of reverse mathematics. Indeed, for any $k \geq 1$, the
statement obtained above in correspondence with $C_k$ is equivalent
to arithmetic comprehension. This means that each $C_k$ with $k \geq
1$ corresponds to \ACA, while it is well-known that $C_k \creds
C_{k+1}$. In other words, at the level of \ACA\ computable analysis
is finer than reverse mathematics.

The correspondence between proof-theoretic and computable
equivalence is more useful when we are at the level of \RCA\ or
\WKL. First, the computable sets are the intended $\omega$-model of
\RCA, which is therefore a formal version of computable mathematics.
Hence we expect that a statement provable in \RCA\ gives rise to a
computable function. Second, we expect most statements equivalent to
\WKL\ to give rise to computably equivalent uncomputable functions.

Sometimes these expectations are fulfilled, and some reverse
mathematics proofs even translate naturally into a computable
analysis proof. This is the case with Theorems \ref{Sep/Path} and
\ref{reversal} below. However the existence of this translation
cannot be taken for granted, and for each direction of the
correspondence we will give examples of failures. In other words, no
automatic translation from the reverse mathematics literature into
computable analysis, or vice versa, is possible. This phenomenon is
a consequence of the different methods and goals of the two
approaches. On one hand, the subsystems of second order arithmetic
studied in reverse mathematics uses freely classical principles with
no algorithmic content, such as excluded middle and proofs by
contraposition. On the other hand, the algorithms of computable
analysis assume the existence of the objects they have to compute,
without the need of proving it. The examples of failure of the
correspondence below highlight these differences.

\subsection{Success of the correspondence}
An often-used equivalent of \ACA\ is the statement that the range of
every one-to-one function from $\N$ to $\N$ is a set. Using the
approach described above this translates into the following
function.

\begin{definition}[\Ran]
Let $\Ran: \sbsq \Bai \to \Can$ be the function that maps any
one-to-one function to the characteristic function of its range,
i.e.
\[
\Ran(p) (n) = \left\{
                \begin{array}{ll}
                  1, & \text{if $\exists m\, p(m) =n$;} \\
                  0, & \text{otherwise.}
                \end{array}
              \right.
\]
for every injective $p: \N \to \N$ and every $n$.
\end{definition}

As expected, we have the following Lemma.

\begin{lemma}\label{Ran}
$\Ran \ceq C_1$.
\end{lemma}
\begin{proof}
First we show $\Ran \cred C_1$. Given $p \in \dom(\Ran)$ let $H(p)
\in \Can$ be defined by $H(p) (\la n,m\ra) = 1$ if and only if $p(m)
= n$. $H: \sbsq \Bai \to \Can$ is computable and it is immediate
that $\Ran(p) (n) = 1- C_1(H(p))(n)$ for every injective $p$ and
$n$.\smallskip

We now show that $C_1 \cred \Ran$. Given $p \in \Bai$ let $H(p) \in
\Bai$ be defined by
\[
H(p) (\la n,m\ra) = \left\{
        \begin{array}{ll}
         \la n,0\ra & \text{if $p(\la n,m\ra) \neq 0$ and $\forall k<m\, p(\la n,k\ra)=0$;}\\
         \la n,m+1\ra & \text{otherwise.}
        \end{array} \right.
\]
The function $H: \Bai \to \Bai$ is computable with $\ran(H)
\subseteq \dom(\Ran)$ (i.e.\ each $H(p)$ is one-to-one). Moreover
$C_1 (p)(n) = 1 - \Ran(H(p)) (\la n,0\ra)$.
\end{proof}

A basic example of reverse mathematics deals with the existence of
least upper bounds of bounded sequences of real numbers. Indeed,
this mathematical principle turns out to be equivalent to \ACA\
(\cite[Theorem III.2.2]{sosoa}). We now show how this equivalence
translates into computable analysis.

\begin{definition}[\Sup]
Let $\Sup: I^\N \to I$ be the function that maps any sequence in
$I^\N$ to its least upper bound.
\end{definition}

\begin{theorem}\label{sup}
$C_1 \ceq \Sup$.
\end{theorem}
\begin{proof}
We start by showing that $\Sup \cred C_1$. Given $(x_n) \in I^\N$
observe that it is easy to use $C_1$ to compute the (characteristic
function of the) set $A = \set{\a \in \Q}{\exists n\, \a < x_n}$.
Now we can computably define a sequence of rationals $(\a_k)$, where
$\a_k = \frac i{2^k}$ is such that $\frac i{2^k} \in A$ and $\frac
{i+1}{2^k} \notin A$. Clearly $(\a_k)$ is a Cauchy representation of
the real number $\Sup (x_n)$.\smallskip

By Lemma \ref{Ran}, to prove $C_1 \cred \Sup$ it suffices to show
that $\Ran \cred \Sup$. Given $p \in \dom(\Ran)$, define $(x_m) \in
I^\N$ by setting $x_m = \sum_{k \leq m} 2^{-(p(k)+1)}$. Given $x =
\Sup(x_m)$ we can define $q: \N \to \N$ by letting $q(n)$ to be the
least $k$ satisfying $x -x_k < 2^{-(n+1)}$. Then for every $n$ we
have
\[
\exists m\, p(m) = n \iff \exists m \leq q(n)\, p(m)=n.
\]
and hence
\[
\Ran (p)(n)= \left\{
    \begin{array}{ll}
        1 & \text{if $\exists m \leq q(n)\, p(m)=n$;}\\
        0 & \text{otherwise.}\\
\end{array} \right.
\]
This shows that, after using \Sup\ to obtain $x$, we can establish
whether $n \in \Ran (p)$ by first computing $q(n)$ by search, and
then checking finitely many values of $m$.
\end{proof}

\subsection{Failure of the correspondence}
We now exhibit some examples where the correspondence outlined above
fails. We first show that sometimes functions arising from
statements provable in \RCA\ are incomputable.

\begin{ex}\label{ex:Omega}
The following function, known as the
\lq\lq\emph{Allwissenheitsprinzip}\rq\rq\ (Principle of
Omniscience), has been studied in detail from the viewpoint of
computable analysis (\cite{Stein,Myl}).

Let $\Omega: \Bai \to \{0,1\}$ be defined by
\[
\Omega(p) = \left\{ \begin{array}{ll}
        0 & \text{if $p=\bar{0}$;}\\
        1 & \text{otherwise.}\\
\end{array} \right.
\]
The incomputability of $\Omega$ follows immediately from Lemma
\ref{comp->cont}. On the other hand, the statement corresponding to
$\Omega$ is
\[
\forall p \in \Bai\, \exists i \in \{0,1\} (i=0 \iff \forall n\,
p(n) =0),
\]
which is obviously provable in \RCA\ (and indeed just from the
excluded middle, except for the coding of functions in the language
of second order arithmetic).
\end{ex}

We now give another example, which is more mathematical, but again
has its roots in the use of classical logic in reverse mathematics.

\begin{ex}\label{ex:Sel}
Let $\mcA_-(\Can)$ be the hyperspace of closed subsets of \Can\
represented by negative information (see Definition
\ref{repr:closed} below) and $\Sel: \sbsq \mcA_-(\Can) \toto \Can$
be the multi-valued function which selects a point from nonempty
closed subsets of \Can. In other words $\Sel (A) = A$, but on the
left-hand side of this equality $A$ is a closed set (and hence a
single element in the hyperspace), while on the right-hand side it
is a set of points in the space \Can.

The statement corresponding to \Sel\ is $\forall A \in \mcA_-(\Can)
(A \neq \emptyset \implies \exists x\, x \in A)$, which is a
tautology, since $A \neq \emptyset$ is an abbreviation for $\exists
x\, x \in A$, and hence provable in \RCA. On the other hand it
follows from Theorem \ref{sel} below that, if we represent closed
sets with respect to negative information (coherently with the
reverse mathematics definition of closed set), $\Sel \ceq \Sep$ and
hence \Sel\ is incomputable.
\end{ex}

\begin{ex}\label{intermediate}
It is well-known that the intermediate value theorem is not
constructive, and it can be shown that the corresponding
multi-valued function is not computable (Brattka and Gherardi have
forthcoming results about the incomputability strength of this
function). On the other hand, a standard proof of the intermediate
value theorem which uses the excluded middle can be carried out in
\RCA\ (\cite[Theorem II.6.6]{sosoa}).
\end{ex}

We now give an example of the opposite phenomena, i.e.\ a theorem
which is not provable in \RCA\ but corresponds to a computable
function.

\begin{ex}\label{compact}
The Heine-Borel compactness of the interval $I$ is Example
(\ref{ex:HBc}) at the beginning of this Section. In reverse
mathematics it is well-known that this statement is equivalent to
\WKL\ (\cite[Theorem IV.1.2]{sosoa}). On the other hand in
computable analysis it is well-known that the function which maps
each countable open covering of $I$ consisting of intervals with
rational endpoints to a finite subcovering is computable
(\cite{Weih00}). We sketch the proof, to emphasize the difference
between the reverse mathematics and the computable analysis
approaches in this case.

There exists a computable enumeration $(\mathfrak{C}_n)$ of all
finite open coverings of $I$ consisting of intervals with rational
endpoints (in \RCA\ we can even prove, e.g.\ using the ideas of the
last part of the proof of Lemma \ref{product} below, that the set of
all these finite open coverings does exist). If we are given an
(infinite) open covering $(U_k)$ of $I$, where each $U_k$ is an
interval with rational endpoints, it suffices to search for $j, n \in
\N$ such that any interval in $\mathfrak{C}_n$ is $U_k$ for some $k
\leq j$. Then $\la U_k: k \leq j \ra$ is the desired finite
subcovering.

In this proof our knowledge of the compactness of $I$ insures that
the search will sooner or later succeed. From the reverse
mathematics viewpoint, the algorithm can be defined in \RCA, but the
proof of its termination requires \WKL.
\end{ex}

\section{The multi-valued function \Sep}\label{section:Sep}

For the reader's convenience, we repeat here the definition of \Sep\
given in the introduction.

\begin{definition}[\Sep]\label{def:sep}
Let $\Sep: \sbsq \Bai \times \Bai \toto \Can$ be defined by
$\dom(\Sep) = \set{(p,q) \in \Bai \times \Bai}{\forall n\, \forall
m\, p(n) \neq q(m)}$,
\[
\Sep (p,q) = \set{r \in \Can}{\forall n (r(p(n)) = 0 \land r(q(n))
= 1)}.
\]
Thus $\Sep (p,q)$ is the set of the characteristic functions of the
sets separating $\ran(p)$ and $\ran(q)$.
\end{definition}

\subsection{\Sep, \Path\ and other incomputable functions}
The following fact follows from standard facts in computability
theory ($\Omega$ was defined in Example \ref{ex:Omega}).

\begin{lemma}\label{Sep:incomp}
$\Sep \ncred \Omega$.
\end{lemma}
\begin{proof}
We show that $\Sep \ncred f$ for any $f: \sbsq X \to \N$. Towards a
contradiction, suppose $\Sep \cred f$ and let $h: \sbsq \Bai \times
\Bai \toto X$ and $k: \sbsq (\Bai \times \Bai) \times \N \toto \Can$
be such that $k((p,q), (f \circ h) (p,q)) \subseteq \Sep(p,q)$ for
every $(p,q) \in \dom(\Sep)$. In particular this holds for
$(p_0,q_0)$, where $p_0,q_0 \in \Bai$ are computable functions with
disjoint, yet computably inseparable, ranges. Since $(f \circ h)
(p_0,q_0) \subseteq \N$, to compute an element of $\Sep(p_0,q_0)$ we
can give as input to $k$ the pair $((p_0,q_0), n)$ for some $n \in
(f \circ h) (p_0,q_0)$. The resulting characteristic function is
computable, a contradiction.
\end{proof}

\begin{cor}
\Sep\ is not computable.
\end{cor}
\begin{proof}
It is easy to see that $f \cred \Omega$ for all computable
multi-valued functions $f$.
\end{proof}

On the other hand, \Sep\ is computably reducible to $C_1$ (we will
show in Corollary \ref{c1path} that $\Sep \creds C_1$).

\begin{lemma}\label{PC}
$\Sep \cred C_1$.
\end{lemma}
\begin{proof}
We define the computable function $h: \Bai \times \Bai \to \Bai$ by
\[
h(p,q) (\la n,m\ra) = \left\{ \begin{array}{ll} 1 & \text{if $p(m)=n$;}\\
0 & \textrm{otherwise.}\\
\end{array} \right.
\]
When $(p,q)\in \dom(\Sep)$ it is immediate that $C_1 (h(p,q)) \in
\Sep(p,q)$. In fact
\[
C_1(h(p,q)) (n)= \left\{ \begin{array}{ll} 0 & \text{if $n \in \ran(p)$;}\\
1 & \textrm{otherwise.}\\
\end{array} \right.\qedhere
\]
\end{proof}

We intend to use computable reducibility to \Sep\ as a way of
assessing incomputability of other functions.

\begin{definition}[\Sep-computable and
\Sep-complete]\label{sepcompcompl} Let $(X,\s_X),(Y,\s_Y)$ be
represented spaces. Then a multi-valued (possibly single-valued)
function $f: \sbsq X\toto Y$ is \emph{\Sep-computable} if $f \cred
\Sep$, \emph{\Sep-complete} if $f \ceq \Sep$.
\end{definition}

To study \Sep\ we introduce the function \Path, which corresponds to
Weak K\"{o}nig's Lemma, i.e.\ the statement asserting the existence of
infinite paths in any infinite binary tree.

\begin{definition}[\Path]
Let $\InfTr \subseteq \mathcal{P} (\Seqd)$ be the set of all
infinite binary trees:
\[
\InfTr = \set{T \subseteq \Seqd}{\forall t \in T\, \forall s
\sqsubseteq t\, s \in T \land \forall n\, T \cap 2^n \neq
\emptyset}.
\]
Let $\Path: \mathcal{P} (\Seqd) \toto 2^\N$ with $\dom(\Path) =
\InfTr$ be the multi-valued function mapping each infinite binary
tree to the set of its infinite paths:
\[
\Path(T) = \set{q \in \Can}{\forall n\, q[n] \in T}.
\]
\end{definition}

The proof of the next theorem follows closely the proof of
\cite[Lemma IV.4.4]{sosoa}.

\begin{theorem}\label{Sep/Path}
$\Sep \ceq \Path$.
\end{theorem}
\begin{proof}
We start by showing that $\Sep \cred \Path$. Let $h: \Bai \times \Bai
\to \mathcal{P} (\Seqd)$ be defined by
\begin{multline*}
h(p,q) = \{\,t \in \Seqd | \\
\forall i<|t| [(\exists j<|t|\, p(j)=i \implies t(i) = 0) \land
(\exists j<|t|\, q(j)=i \implies t(i)=1)]\,\}.
\end{multline*}
The function $h$ is clearly computable.

If $(p,q) \in \dom(\Sep)$ it is easy to see that $h(p,q) \in \InfTr$
and any infinite path in $h(p,q)$ is the characteristic function of
a set separating $\ran(p)$ and $\ran(q)$. Thus $\Path(h(p,q))
\subseteq \Sep(p,q)$ for every $(p,q) \in \dom(\Sep)$, showing $\Sep
\cred \Path$.\smallskip

We now prove $\Path \cred \Sep$. Given any $T \in \mathcal{P}
(\Seqd)$ let, for $s \in \Seqd$ and $i<2$,
\begin{align*}
\theta_T (n,s) \iff & \exists t \in 2^n (t \in T \land s \sqsubseteq t);\\
\varphi_T (s,i) \iff & \exists n (\theta_T (n, s*i) \land \neg \theta_T (n, s*(1-i))).
\end{align*}
Notice that if $T \in \InfTr$ we have $\neg (\varphi_T (s,0) \land
\varphi_T (s,1))$ for all $s \in \Seqd$.

It is easy to define a computable function $h: \mathcal{P} (\Seqd)
\to \Bai \times \Bai$ such that $h(T) = (p_T,q_T)$ with
\begin{align*}
\ran(p_T) = & \set{s+2}{\varphi_T (s,0)} \cup \{0\} \quad \text{and}\\
\ran(q_T) = & \set{s+2}{\varphi_T (s,1)} \cup \{1\}.
\end{align*}
If $T \in \InfTr$ the observation above implies $\ran(p_T) \cap
\ran(q_T) = \emptyset$, i.e.\ $h(T) \in \dom(\Sep)$.

Given $r \in \Can$ we can recursively define $k(r) \in \Can$ by
\[
k(r)(m) = r(k(r)[m]+2).
\]
We have thus defined a computable function $k: \Can \to \Can$.

If $T \in \InfTr$ and $r \in \Sep(h(T))$ we show by induction on $m$
that $\theta_T (n,k(r)[m])$ holds for all $m$ and $n \geq m$. To
simplify the notation, let $s_m = k(r)[m]$.
\begin{itemize}
\item we have $s_0 = \l$ (recall that $\l$ is the empty
    sequence), and $\theta_T (n,\l)$ for all $n \geq 0$ follows
    immediately from the fact that $T$ is an infinite tree;
\item now suppose that $\theta_T (n,s_m)$ holds for all $n \geq
    m$. We want to show that $\theta_T (n,s_{m+1})$ holds for
    all $n>m$.

If $\varphi_T (s_m,0)$ then $s_m +2 \in \ran(p_T)$ and from $r
\in \Sep(h(T))$ it follows that $r(s_m +2) = 0$. Therefore
$s_{m+1} = s_m*0$. Let $N \in \N$ be such that $\theta_T
(N,s_m*0)$ and $\neg \theta_T (N,s_m*1)$. For $n \geq N$, it
cannot be $\theta_T (n,s_m*1)$ (because $T$ is a tree), but by
induction hypothesis $\theta_T (n,s_m)$ holds. Hence we have
$\theta_T (n,s_m*0)$, that is $\theta_T (n,s_{m+1})$. When $m
\leq n < N$, $\theta_T (n,s_{m+1})$ follows from $\theta_T
(N,s_{m+1})$ and $T$ being a tree.

When $\varphi_T (s_m,1)$, the argument is similar to the previous
case.

If $\varphi_T (s_m,0)$ and $\varphi_T (s_m,1)$ both fail, then
for every $n>m$ either
\begin{gather*}
\neg \theta_T (n,s_m*0) \land \neg \theta_T (n,s_m*1)\\
\text{or} \quad \theta_T (n,s_m*0) \land \theta_T (n,s_m*1).
\end{gather*}
The first case is impossible, since $\theta_T (n,s_m)$ for all
$n \geq m$. Therefore only the second case is possible, which
means that no matter what is $s_{m+1}$ (i.e., whatever is the
value of $r(s_m +2)$) we have $\theta_T (n,s_{m+1})$ for all
$n>m$.
\end{itemize}

In particular for all $n$ we have $\theta_T (n,k(r)[n])$ and thus
$k(r)[n] \in T$. Hence $k(r) \in \Path(T)$. We have thus shown that
$k(\Sep(h(T))) \subseteq \Path (T)$, which shows that $\Path \cred
\Sep$.
\end{proof}

We will need to consider also paths in bounded trees. These are the
finitely branching trees for which there is an explicit bound,
depending on the level, for the values attained by the sequences
occurring in the tree.

\begin{definition}[\PathB]
Let $\InfTrB \subseteq \mathcal{P} (\Seq) \times \Bai$ be the set of
all infinite \lq\lq bounded trees\rq\rq. $(T,b) \in \mathcal{P}
(\Seq) \times \Bai$ belongs to \InfTrB\ if and only if
\[
\forall t \in T\, \forall s \sqsubseteq t\, s \in T \land \forall i\,
\forall t \in T \cap \N^{i+1}\, t(i)<b(i) \land \forall n\, T \cap \N^n
\neq \emptyset.
\]
Let $\PathB: \mathcal{P} (\Seq) \times \Bai \toto 2^\N$ with
$\dom(\PathB) = \InfTrB$ be the multi-valued function mapping each
infinite bounded tree to the set of its infinite paths:
\[
\PathB(T,b) = \set{p \in \Bai}{\forall n\, p[n] \in T}.
\]
\end{definition}

The following result is the computable analysis equivalent of Lemma
IV.1.4 in \cite{sosoa}. We omit the proof, which is a
straightforward adaptation of the proof in the reverse mathematics
setting.

\begin{lemma}\label{boundedKL}
$\PathB \ceq \Path$ and hence $\PathB \ceq \Sep$.
\end{lemma}

We now show the incomparability of \Sep\ and $\Omega$. We already
know from Lemma \ref{Sep:incomp} that $\Sep \ncred \Omega$.

\begin{theorem}\label{OmegaPath}
$\Omega \ncred \Sep$.
\end{theorem}
\begin{proof}
By Theorem \ref{Sep/Path} it suffices to show that $\Omega \ncred
\Path$.

Suppose that $\Omega \cred \Path$ and let $h: \Bai \toto \InfTr$ and
$k: \Bai \times \Can \toto \{0,1\}$ be computable multi-valued
functions witnessing this. In other words, $k(p,q) = \Omega(p)$ for
every $q$ such that $q \in \Path(T)$ for some $T \in h(p)$ (on such
pairs $k$ is single-valued).

For $n \in \N$ let $p_n = (\bar{0}[n] * 1) \conc \bar{0} \in \Bai$,
where $\bar{0}$ is the only argument on which $\Omega$ takes value
$0$. Clearly $\lim_n p_n = \bar{0}$, and, since $h$ has a
computable, and hence continuous, realizer, there exist $T \in
h(\bar{0})$ and a sequence of infinite trees $(T_n)$ with $T_n \in
h(p_n)$ such that $\lim T_n = T$.

For any $n \in \N$ let $q_n \in \Path(T_n)$, so that $k(p_n,q_n) =
\Omega(p_n) =1$. Since \Can\ is compact we may assume that $\lim_n
q_n = q$ for some $q \in \Can$. For every $m$, if $n$ is
sufficiently large, $q[m] = q_n[m]$ and $T_n \cap 2^m = T \cap 2^m$,
and hence $q[m] \in T$. Thus $q \in \Path (T)$.

Again, $k$ has a continuous realizer and we should have $\lim_n k
(p_n,q_n) = k(\bar{0},q) =\Omega(\bar{0}) = 0$ which is impossible
since $k(p_n,q_n) = 1$ for all $n$.
\end{proof}

\begin{cor}\label{c1path}
$\Sep \creds C_1$.
\end{cor}
\begin{proof}
Straightforward from Lemma \ref{PC} and the previous theorem, since
$\Omega \cred C_1$.
\end{proof}

\subsection{Iterating \Sep-computable functions}
We now show that iterating \Sep-computable functions does not
increase the degree of incomputability. Thus the situation is quite
different from the case of the $C_i$'s, where $C_i \ceq
\underbrace{C_1 \circ C_1 \circ \dots \circ C_1}_{i \text{ times}}$
and hence (recalling that $C_i \creds C_j$ when $i<j$) $C_i \creds
C_i \circ C_i$ for every $i>0$.

First, we deal with \Path. Actually, since the application of \Path\
to itself is meaningless, we use a computable function to transform
the output of \Path\ into an infinite tree that is given as input to
another application of \Path.

\begin{lemma}\label{doublappl}
Let $f: \sbsq \Can \to \InfTr$ be computable. Then
\[
\Path \circ f \circ \Path \cred \Path.
\]
\end{lemma}
\begin{proof}
For any $p \in \Bai$ let $p_0, p_1\in \Bai$ be such that $p = p_0
\oplus p_1$. Define analogously $t_0$ and $t_1$ when $t \in \Seqd$.
The maps $p \mapsto p_0$ and $p \mapsto p_1$ are obviously
computable.

Given any $T \in \InfTr$ we will computably define $\widetilde{T}
\in \InfTr$ such that if $T \in \dom(\Path \circ f \circ \Path)$ and
$p \in \Path(\widetilde{T})$ we have $p_0 \in \Path(T)$ and $p_1 \in
(\Path \circ f) (p_0)$. This suffices to prove $\Path \circ f \circ
\Path \cred \Path$ (in the notation of Definition
\ref{def:redmulti}, $T \mapsto \widetilde{T}$ and $(T,p) \mapsto
p_1$ play the role of $h$ and $k$ respectively).

Let $\widehat{f}: \Seqd \to \mathcal{P} (\Seqd)$ be the computable
function defined as follows. $\widehat{f}(t)$ is the set of all $s
\in \Seqd$ such that after $|t|$ steps (when at most the first $|t|$
bits of input have been used) in the computation of $f (t \conc q)$
(for any $q \in \Can$) no $v \sqsubseteq s$ has been marked as not
belonging to the output tree.

Notice that $\widehat{f}(t)$ is a tree and $t \sqsubseteq u$ implies
$\widehat{f}(t) \supseteq \widehat{f}(u)$. Moreover for all $p \in
\dom(f)$ and $n \in \N$, $f(p) \subseteq \widehat{f}(p[n])$. Since we
have that if $s \notin f(p)$ then $s \notin \widehat{f}(p[n])$ for
some $n$, $f(p) = \bigcap_n \widehat{f}(p[n])$ for every $p \in
\dom(f)$. Thus we can view $\widehat{f}$ as an approximation of $f$
from above.

Let
\[
\widetilde{T} = \set{t \in \Seqd}{t_0 \in T \land t_1 \in \widehat{f}(t_0)},
\]
so that the map $T \mapsto \widetilde{T}$ is computable. Using the
properties of $\widehat{f}$ mentioned above, it is immediate to
check that $\widetilde{T}$ is a tree. Moreover if $T \in \dom(\Path
\circ f \circ \Path) = \set{T \in \InfTr}{\Path(T) \subseteq
\dom(f)}$ then $\widetilde{T} \in \InfTr$. In fact if $q \in
\Path(T)$ then $f(q) \in \InfTr$ and if $u \in f(q)$ has length $n$
then the sequence $t \in 2^{2n}$ such that $t_0 = q[n]$ and $t_1 =
u$ belongs to $\widetilde{T}$.

If $p \in \Path(\widetilde{T})$ then $p_0 \in \Path(T)$ is immediate.
If $p_0 \in \dom(f)$ and $p_1 \notin \Path(f(p_0))$ then there exist
$m$ such that $p_1[m] \notin f(p_0)$ and hence $n$ such that $p_1[m]
\notin \widehat{f}(p_0[n])$. We may assume $m \leq n$, which implies
$p_1[n] \notin \widehat{f}(p_0[n])$, contradicting $p[2n] \in
\widetilde{T}$. Thus $p_1 \in \Path(f(p_0))$.
\end{proof}

In a similar way, one can prove the following Lemma.

\begin{lemma}\label{ugly}
Let $X$ and $Y$ be represented spaces and suppose that $h: \sbsq X
\toto \InfTr$, $i: \sbsq X \times \Can \toto \InfTr$ and $j: \sbsq X
\times \Can \toto Y$ are computable. Let $\ell: \sbsq X \toto Y$ be
defined by
\[
\ell(x) = \bigcup \set{j(x,q)}{\exists p \in (\Path \circ h) (x)\, q \in (\Path \circ i) (x,p)}.
\]
Then $\ell \cred \Path$.
\end{lemma}

\begin{theorem}\label{compositionSep-comp}
Let $f: \sbsq X \toto Y$ and $g: \sbsq Y \toto Z$ be \Sep-computable
multi-valued functions between represented spaces. Then $g \circ f:
\sbsq X \toto Z$ is \Sep-computable.
\end{theorem}
\begin{proof}
By Theorem \ref{Sep/Path} we have $f,g \cred \Path$ and there exist
computable $h: \sbsq X \toto \InfTr$, $k: \sbsq X \times \Can \toto
Y$, $h': \sbsq Y \toto \InfTr$, $k': \sbsq Y \times \Can \toto Z$
witnessing this as in Definition \ref{def:redmulti}. Therefore
$k'(k(x,p),q) \subseteq (g \circ f)(x)$ for all $x \in \dom (g \circ
f)$, $p \in (\Path \circ h) (x)$, and $q \in (\Path \circ h' \circ
k) (x,p)$.

To use Lemma \ref{ugly} we need to identify the functions involved.
$h$ (in the notation of Lemma \ref{ugly}) is $h$, $i$ is $(x,p)
\mapsto (h' \circ k) (x,p)$ and $j$ is
\[
(x,q) \mapsto \set{k'(k(x,p),q)}{p \in (\Path \circ h) (x) \land q \in
(\Path \circ h' \circ k)(x,p)}.
\]
Then the function $\ell$ of Lemma \ref{ugly} is such that $\dom(g
\circ f) \subseteq \dom(\ell)$ and $\ell(x) \subseteq (g \circ f)
(x)$ for every $x \in \dom(g \circ f)$. Thus $g \circ f \cred \ell$
by Lemma \ref{superset} and $\ell \cred \Path$ by Lemma \ref{ugly}.
Hence $g \circ f \cred \Path$. Theorem \ref{Sep/Path} now implies
that $g \circ f$ is \Sep-computable.
\end{proof}

\section{Banach spaces in computable analysis}\label{section:BS}

\subsection{Effective Banach spaces}
To deal with Banach spaces in the context of computable analysis we
need to give definitions which are analogous to the ones given in
Section \ref{section:ca} for metric spaces.

\begin{definition}[Effective Banach space]
An \emph{effective Banach space} is a triple $(X, \|\phantom{x}\|,
e)$ such that
\begin{itemize}
 \item $X$ is a Banach space with norm $\|\phantom{x}\|$;
 \item $e: \N \to X$ is a \emph{fundamental sequence}, i.e.\ a
     sequence whose linear span is dense in $X$;
 \item $(X,d,a_e)$ is an effective metric space, where $d(x,y) =
     \|x-y\|$ and $a_e (s) = \sum_{i < |s|} a_\Q (s(i)) \cdot
     e(i)$ for $s \in \Seq$.
\end{itemize}
We will always assume that $X$ is nontrivial, i.e.\ that $\|e(i)\|
\neq 0$ for some $i \in \N$.
\end{definition}

Notice that an effective Banach space is separable.

The domain of the multi-valued function corresponding to the
Hahn-Banach Theorem consists of all effective Banach spaces. If this
is interpreted naively, we would need a method to code any possible
effective Banach space. Clearly, there are \lq\lq too many\rq\rq\
such spaces to allow a well defined single-valued representation
and, since the collection of all effective Banach spaces is not even
a set, even a multi-representation approach (in the sense of
\cite{GW}) is questionable.

We can overcome this problem by considering a set which contains all
effective Banach spaces up to isomorphism. For this set we can then
define a single-valued representation. (Notice that this approach is
very close to the one used in reverse mathematics, where it is
customary to represent mathematical objects by \lq\lq codes\rq\rq.)
We will adapt Weihrauch's notion of constructive metric completion
(see \cite{Weih00}) to the case of effective Banach spaces.

\subsection{Constructive Banach completions}\label{subs:cBc}
For every $s \in \Seq$ let
\[
c_s = \sum_{i<|s|} a_\Q(s(i)) \cdot i
\]
where we are viewing the right-hand side as a formal linear
combination of elements of $\N$ with scalars in $\Q$. Let $C =
\set{c_s}{s \in \Seq}$.

We define sum on $C$ and scalar multiplication of an element of $C$
by an element of $\Q$ in the obvious way. A \emph{noted
pseudo-normed space} is then a pair $N=(C, \|\phantom{x}\|)$ such
that $\|\phantom{x}\|: C \to \R$ is a pseudo-norm on $C$, i.e.
\begin{itemize}
    \item $\|c_s\|=0$ whenever $s(i) = 0$ for all $i < |s|$
        (recall that $a_\Q(0) =0$);
    \item $\|c_s+c_t\| \leq \|c_s\| + \|c_t\|$, for all $s,t\in
        \Seq$;
    \item $\|\a \cdot c_s\|=|\a| \cdot \|c_s\|$ for all $s \in
        \Seq$ and $\a \in \Q$.
\end{itemize}
Again, we assume that $\|c_s\| \neq 0$ for some $s \in \Seq$.

The pseudo-norm $\|\phantom{x}\|$ defines a pseudo-metric $d$ over
$C$ as usual by $d(c_s, c_t)=\|c_s - c_t\|$.

We now build the constructive Banach completion of $N$, as a
particular effective Banach space. Let $\widehat{C}$ be the set of
all Cauchy sequences of elements of $C$ which satisfy the usual
effective requirement:
\[
\widehat{C} = \set{(c_{s_i})}{\fa j\, \fa i<j\,
d(c_{s_i}, c_{s_j}) < 2^{-i}}.
\]
Define an equivalence relation $\sim$ on $\widehat{C}$ by
\[
(c_{s_i}) \sim (c_{t_i}) \Otto \lim d(c_{s_i}, c_{t_i}) =0,
\]
and notice that this condition is equivalent to $\fa i\, d(c_{s_i},
c_{t_i}) \leq 2^{-(i-1)}$. We denote by $[c_{s_i}]_{i \in \N}$ the
$\sim$-equivalence class of $(c_{s_i})$. We introduce then the
linear operations on $\widehat{C}/{\sim}$ by
\begin{gather*}
[c_{s_i}]_{i \in \N} + [c_{t_i}]_{i \in \N} = [c_{s_{i+1}} +
c_{t_{i+1}}]_{i \in \N};\\
a \cdot [c_{s_i}]_{i \in \N} = [a_\Q (n_{k+i}) \cdot c_{s_{k+i}}]_{i
\in \N}
\end{gather*}
where $a \in \R$, $(a_\Q(n_i))$ is a Cauchy sequence effectively
converging to $a$, and $k$ is such that $|a_\Q(n_0)| + \|c_{s_0}\|
+2 < 2^k$. We leave to the reader checking that these definitions
are meaningful and make $\widehat{C}/{\sim}$ a vector space (some of
the details are spelled out in \cite[p.75]{sosoa}).

We further define
\[
\|[c_{s_i}]_{i \in \N}\|_{\widehat{C}/{\sim}} = \lim \|c_{s_i}\|
\]
and one can check that $\widehat{C}/{\sim}$ is a Banach space.
Notice that $d_{\widehat{C}/{\sim}} ([c_{s_i}]_{i \in \N},
[c_{t_i}]_{i \in \N}) = \lim d(c_{s_i}, c_{t_i})$.

Define $e: \N \to \widehat{C}/{\sim}$ by $e(n) = [c_{\bar{0}[n]
*1}]_{i \in \N}$ (recall that $a_\Q(0) =0$ and $a_\Q(1) =1$) where
on the right hand side we have a constant sequence. The triple
$(\widehat{C}/{\sim}, \|\phantom{x}\|_{\widehat{C}/{\sim}}, e)$ is
an effective Banach space, the \emph{constructive Banach completion}
of the noted pseudo-normed space $N$.

The function $c_s \mapsto [c_s]_{i \in \N}$ maps $C$ into
$\widehat{C}/{\sim}$ respecting the vector operations and the
(pseudo-)norm. Therefore we can view $C$ as the linearly closed dense
subspace of $\widehat{C}/{\sim}$ generated by the fundamental
sequence $e$.

\begin{definition}[The space of all effective Banach spaces]
Let \BS\ be the set of all constructive Banach completions. This set
contains all effective Banach spaces up to isomorphism, and we
consider it as the space of all effective Banach spaces.

Consider now the second countable $T_0$-topology on \BS\ with
sub-basis given by the sets of the form
\[
U_{\la i,s,t,j \ra} = \set{(\widehat{C}/{\sim}, \|\phantom{x}\|_{\widehat{C}/{\sim}}, e)}
{a_\Q(i) < d_{\widehat{C}/{\sim}} (a_e(s), a_e(t)) < a_\Q(j)}.
\]
This topology on \BS\ is associated with the standard representation
$\d_{\BS}: \sbsq \Bai \to \BS$ defined by $\d_{\BS}(p)=
(\widehat{C}/{\sim}, \|\phantom{x}\|_{\widehat{C}/{\sim}}, e)$ if
and only if $p$ enumerates the set
\[
\set{\la i,s,t,j\ra}
{(\widehat{C}/{\sim}, \|\phantom{x}\|_{\widehat{C}/{\sim}}, e) \in U_{\la i,s,t,j\ra}}.
\]
\end{definition}

We often write $(X, \|\phantom{x}\|, e) \in \BS$, or simply $X \in
\BS$, in place of $(\widehat{C}/{\sim},
\|\phantom{x}\|_{\widehat{C}/{\sim}}, e) \in \BS$, but we always
understand that the construction of $X$ as a constructive Banach
completion uniquely determines both the norm and the fundamental
sequence.

An element in $\BS$ with a computable name is a computable Banach
space in the sense of \cite{Bra08}. Since we view \BS\ as the space
of all effective Banach spaces, we view the subset of its computable
elements as the set of all computable Banach spaces.

\subsection{Representations of closed and compact sets and of linear bounded functions}
We recall some representations of closed and compact subsets of
metric spaces which have been widely used in the literature (see
e.g. \cite{BP}).

\begin{definition}[Representations of closed
sets]\label{repr:closed} For an effective metric space $X$ we denote
by $\mcA_+(X)$ and $\mcA_-(X)$ the hyperspace of closed subsets of
$X$ viewed respectively with representations $\p^X_+$ and $\p^X_-$,
where
\begin{itemize}
 \item $\p^X_+ (p) =A$ if and only if $p_i \in \dom(\d_X)$ for
     all $i \in \N$ (where $p_i(j) = p(\la i,j \ra)$) and $A =
     \overline{\set{\d_X(p_i)}{i \in \N}}$;
 \item $\p^X_- (p) =A$ if and only if $X \setminus A = \bigcup
     B^X_{p(i)}$ (recall that $\{B^X_n\}$ enumerates all
     rational open balls in $X$).
\end{itemize}
\end{definition}

In the reverse mathematics literature, the elements of $\mcA_+(X)$
and of $\mcA_-(X)$ are called respectively separably closed sets and
closed sets.

\begin{definition}[Representations of compact sets]
For an effective metric space $X$ we denote by $\mcK(X)$ and
$\mcK_-(X)$ the hyperspace of compact subsets of $X$ viewed
respectively with representations $\k^X$ and $\k^X_-$, where
\begin{itemize}
    \item $\k^X (p) = K$ if and only if $p$ enumerates
        \[
        \textstyle{\set{s}{K \subseteq \bigcup_{i < |s|} B^X_{s(i)}
        \land \fa i < |s|\, K \cap B^X_{s(i)} \neq \eps};}
        \]
    \item $\k^X_-(p) = K$ if and only if $p$ enumerates
        \[
        \textstyle{\set{s}{K \subseteq \bigcup_{i < |s|} B^X_{s(i)}}.}
        \]
\end{itemize}
\end{definition}

We are now in a position to define the domain of the multi-valued
function corresponding to the separable Hahn-Banach theorem. In
doing so, we borrow an idea from \cite{Weih01}: to denote a partial
continuous function with closed domain $f$ we employ a realizer of
$f$ and a name for $\dom(f)$ with respect to a representation of the
hyperspace of closed sets. Moreover, we further generalize and
consider closed subsets of arbitrary elements of \BS.

\begin{definition}[Space of partial linear bounded functionals]
Let \PF\ be the set of all quadruples $(X,A,f,r)$ (usually written
$f_{(X,A,r)}$) such that
\begin{itemize}
   \item $X \in \BS$;
   \item $A$ is a closed linear subspace of $X$;
   \item $f: A \to \R$ is linear and bounded;
   \item $r = \|f\| \in \R$ (recall that the norm $\|f\|$ is defined by
$\|f\|=\sup\set{|f(x)|}{x \in A \land \|x\|=1}$).
\end{itemize}
The representation of \PF\ is defined by $\d_{\PF} (p) =
f_{(X,A,r)}$ if and only if
\begin{itemize}
    \item $\d_{\BS}(p_0)=X$;
    \item $\p^X_+ (p_1)= A$;
    \item $\e(p_2)$ is a realizer of $f$;
    \item $\d_\R(p_3)=r$
\end{itemize}
(the $p_i$'s were defined in Definition \ref{repr:closed}).
\end{definition}

\section{The Hahn-Banach Theorem}\label{section:HB}

\subsection{The multi-valued function \HB}
We now come to the question of the computational complexity of the
Hahn-Banach Theorem. We start by giving the formal definition of the
Hahn-Banach multi-valued function.

\begin{definition}[Hahn-Banach multi-valued function]\label{HB}
Let $\HB: \sbsq \PF \toto \PF$ be the multi-valued function with
$\dom(\HB) = \{f_{(X,A,1)} \in \PF\}$ defined by
\[
\HB(f_{(X,A,1)}) = \set{g_{(X,X,1)}}{g \restriction A =f}.
\]
\end{definition}

For any computable normed space $X$, and in particular for any
computable Banach space, Brattka (\cite{Bra08}) first proves a
computable version of the Banach-Alaoglu Theorem. Then he shows that
for any computable Banach space there is a \SI02-computable
multi-valued function that maps $f$ to the extensions $g$ which
satisfy the requirements of the Hahn-Banach Theorem (although the
notion of \SI02-computable multi-valued function is not explicitly
used in \cite{Bra08}). We will use the same ideas to show that \HB\
is $\Sep$-computable, but some fundamental modifications are
necessary.

First, we point out that Brattka's proof is not uniform, since it
breaks up into two cases, depending on whether the dimension of the
normed space $X$ is finite or infinite. Even for countable vector
spaces over $\Q$, the function establishing whether the space is
finite-dimensional is not computable, and indeed not even
$\Sep$-computable\footnote{\label{Q}To see this, let $Q = \set{w \in
\Q^{<\N}}{|w| = 0 \lor w(|w|-1) \neq 0}$. We view $Q$ as a vector
space over $\Q$ in the obvious way, and let $\Vect = \set{V
\subseteq Q}{V \text{ is a vector space}}$. Let $Dim: \Vect \to 2$
be defined by
\[
Dim(V)=
\begin{cases}
0 & \text{if $\dim(V) = \infty$;}\\
1 & \text{if $\dim(V) < \infty$.}
\end{cases}
\]
Define the computable function $V: \Bai \to \Vect$ by $V(p) = \set{w
\in Q}{\forall i < |w|\, p(i)=0}$. Then $Dim \circ V = \Omega$ and
thus $\Omega \cred Dim$.}. Since we are interested in evaluating the
complexity of a multi-valued function which takes in input any
possible effective Banach space, we need to get rid of this
dichotomy. We will thus give a uniform structure to Brattka's proof,
also simplifying some steps along the way.

\subsection{Selecting points in closed subsets of compact sets}
Brattka's proof uses a multi-valued choice function on compact sets
to select elements in the set $H(f)$ of all extensions of $f$ (this
is the \SI02-computable step in that proof). Actually, in this
approach one needs to consider $H(f)$ as a \emph{compact} subset of
a compact space $\widehat{X}$. We do not need this step, since the
simpler property of being \emph{closed} in the compact set
$\widehat{X}$ is enough to apply a selection multi-valued function
which is $\Sep$-computable, by Theorem \ref{sel} below.

Although not necessary to our main goal, in Theorem \ref{sel} we
also formulate a general condition of \Sep-completeness for this
selection problem. To achieve this, we recall the following notion
already used in \cite{BraGhe07,BraGhe08}.

\begin{definition}[Richness]
A computable  metric space $X$ is \emph{rich}, or \emph{computably
uncountable}, if there is a computable injective map $\i: \Can
\hookrightarrow X$.
\end{definition}

It is known that if $\i$ is as above then also its partial inverse
$\i^{-1}$ is computable, and thus $\i$ is a computable embedding.

\begin{theorem}\label{sel}
For a computable metric space $(X,d,a)$, let $\Sel_{\mcK(X)}: \sbsq
\mcK(X) \times \mcA_-(X) \toto X$ be the multi-valued function with
domain $\set{(K,A)}{\eps \neq A \subseteq K}$ and
\[
\Sel_{\mcK(X)} (K,A) = A
\]
(where on the left-hand side $A$ is a member of the hyperspace of
the closed subsets of $X$, while on the right-hand side is a set of
points). Thus, $\Sel_{\mcK(X)}$ is the multi-valued function which
selects a point from a nonempty closed subset of a compact subset of
$X$. Then
\begin{enumerate}[(1)]
    \item $\Sel_{\mcK(X)}$ is \Sep-computable;
    \item if $X$ is rich then $\Sel_{\mcK(X)}$ is \Sep-complete.
\end{enumerate}
\end{theorem}
\begin{proof}
(1) Given $K \in \mcK(X)$ we can uniformly obtain $q \in \Bai$ and
an infinite sequence of finite sequences $(\la x^n_j\ra _{j <
q(n)})$ of elements of $X$ such that for every $n \in \N$ we have $K
\subseteq \bigcup_{j < q(n)} B^X (x^n_j; 2^{-n})$.

For $A \in \mcA_-(X)$ such that $\eps \neq A \subseteq K$, we can
uniformly obtain sequences $(b_i)$ in $\ran(a)$ and $(\a_i)$ in $\Q$
such that $X \setminus A = \bigcup_{i\in \N} B^X (b_i; \a_i)$. We
select an element of $A$ by approximating points which do not belong
to any $B^X (b_i;\a_i)$. More precisely, we construct a tree $T =
T(K,A) \subseteq \Seq$ by letting $s \in T$ if and only if
\begin{itemize}
  \item $\fa n<|s|\, s(n) < q(n)$;
  \item $\fa n,i,k < |s|\, d(x^n_{s(n)}, x^i_{s(i)})_{[k]} \leq
      2^{-n} + 2^{-i} +2^{-k}$;
  \item $\fa n,i,k < |s|\, d(x^n_{s(n)}, b_i)_{[k]} \geq \a_i -
      2^{-n} - 2^{-k}$
\end{itemize}
where for $a \in \R$, $a_{[k]}$ is a rational approximation within
$2^{-k}$ of $a$.

Notice that, since $A \neq \eps$, $(T,q) \in \InfTrB$. For all $p\in
\PathB(T,q)$ we have that $x = \lim x^n_{p(n)}$ exists, is
computable from $p$, and does not belong to any $B^X (b_i;\a_i)$.
Hence $x \in A$. This gives $\Sel_{\mcK(X)} \cred \PathB$. By Lemma
\ref{boundedKL} we have $\Sel_{\mcK(X)} \cred \Sep$.

(2) By Theorem \ref{Sep/Path} it suffices to show $\Path \cred
\Sel_{\mcK(X)}$ when $X$ is rich.

First, we show $\Path \cred \Sel_{\mcK(\Can)}$. For $T \in \InfTr$
define
\[
A_T = \Can \setminus \bigcup \set{B^\Can (t \conc \bar{0}; 2^{-(|t|-1)})}{t \notin T} \subseteq \Can.
\]
Since $B^\Can (t \conc \bar{0}; 2^{-(|t|-1)}) = \set{p \in \Can}{t
\sqsubseteq p}$ we have $\Path(T) = \Sel_{\mcK(\Can)} (\Can, A_T)$.
Since the map $T \mapsto (\Can, A_T)$ from $\InfTr$ to $\mcK(\Can)
\times \mcA_-(\Can)$ is computable, $\Path \cred \Sel_{\mcK(\Can)}$.

If $X$ is rich, let $\i: \Can \hookrightarrow X$ be a computable
injection. As observed in \cite{BraGhe07,BraGhe08}, $\ran(\i) \in
\mcK (X)$. By the proof of the Embedding Theorem of
\cite{BraGhe07,BraGhe08}, the map from $\mcA_- (\Can)$ to $\mcA_-
(X)$ which sends $A$ to $\i(A)$ is computable. Hence for every $A
\in \mcA_- (\Can)$ we have
\[
(\i^{-1} \circ \Sel_{\mcK(X)}) (\i(\Can), \i(A)) = \Sel_{\mcK(\Can)} (\Can, A).
\]
Using the notation of the first part of this proof, we have
\[
(\i^{-1} \circ \Sel_{\mcK(X)}) (\i(\Can), \i(A_T)) = \Path(T)
\]
for every $T \in \InfTr$. This shows $\Path \cred \Sel_{\mcK(X)}$.
\end{proof}

\subsection{Proof of $\HB \cred \Sep$}
Brattka's proof uses the Effective Independence Lemma of Pour-El and
Richards (\cite[p.142]{PR}). To make Brattka's argument uniform we
need a uniform version of that result.

\begin{lemma}[Uniform Effective Independence Lemma]\label{ueil}
For all $(X, \|\phantom{x}\|, e) \in \BS$ there exists $q \in \Bai$
such that, letting $R = \set{j>0}{q(j) = q(0)}$, $q$ restricted to
$\N \setminus R$ is one-to-one and $\set{(e \circ q) (j)}{j \in \N
\setminus R}$ is a (possibly finite) linearly independent set whose
linear span is dense in $X$.

Let $\zeta: \BS \toto \Bai$ be the multi-valued function such that
$\zeta(X, \|\phantom{x}\|, e)$ is the set of all $q$ satisfying the
condition above. Then $\zeta$ is computable.
\end{lemma}
\begin{proof}
We prove at once both statements of the Lemma by defining a
computable realizer for $\zeta$. To this end we construct, uniformly
in a name for $X \in \BS$, $q$ by stages. We will also keep track of
$R$ by letting $R_n = \set{0 < j \leq n}{q(j) = q(0)}$. Let $N$ be
such that $\|e(N)\| \neq 0$.

At stage 0 we let $q(0) = N$ and $R_0 =\eps$.

At stage $n+1$ we suppose to have defined $q(0), \dots, q(n)$ and
$R_n \subseteq \{1, \dots, n\}$ such that
\begin{itemize}
\item $q(j) = q(0) = N$ for all $j \in R_n$;
\item $T_n = \set{(e \circ q) (j)}{j \leq n \land j \notin R_n}$
    is linearly independent.
\end{itemize}
We let $T_n = \{v_1, \dots, v_k\}$ (obviously $k \leq n$).

For every $i \leq n+1$ we run a test, described below, which stops
after a finite amount of time with answer either (a) or (b). If the
answer is (a) then we are sure that $T_n  \cup \{e(i)\}$ is linearly
independent. If the answer is (b) then $e(i)$ can be approximated
within $2^{-(n+1)}$ by a rational linear combination of elements of
$T_n$. Therefore, if for some $i$ the answer is $(b)$ at every stage
$\geq i$, then actually $e(i)$ belongs to the closure of the linear
span of $T = \bigcup_{n\in \N} T_n$.

The test is based on the following fact, proved in \cite[p.143]{PR}.
For $m, \ell \in \N$, let $S_{m,\ell}$ be the set of all of all $\la
\b_1, \dots, \b_\ell\ra \in \Q^\ell$ such the denominators of $\b_1,
\dots, \b_\ell$ are $2^{m}$ and $1 \leq |\b_1|^2 + |\b_2|^2 + \dots
+ |\b_\ell|^2 \leq 4$. (The $S_{m,\ell}$'s are finite and can be
uniformly computably enumerated in $m$ and $\ell$.) Pour-El and
Richards prove that a finite subset $\{w_1, \dots, w_\ell\}$ of a
Banach space is linearly independent if and only if for some $m \geq
2 \ell$
\[
\min \set{\|\b_1 w_1+ \dots + \b_\ell w_\ell\|}{\la \b_1, \dots, \b_\ell\ra
\in S_{m,\ell}} > 2^{-m} \cdot (\|w_1\|+ \dots + \|w_\ell\|).
\]

Given $i \leq n+1$ the test alternatively searches
\begin{enumerate}[\quad (a)]
  \item for $m \geq 2(k+1)$ such that
\begin{multline*}
\min \set{\|\b_1 v_1+ \dots + \b_k v_k + \b_{k+1} e(i)\|}{\la
\b_1, \dots, \b_{k+1}\ra \in S_{m,k+1}} > \\
2^{-m} \cdot (\|v_1\| + \dots + \|v_k\| + \|e(i)\|),
\end{multline*}
  \item and for $\g_1, \dots, \g_k \in \Q$ such that
\[
\|e(i)-(\g_1 v_1+ \dots + \g_k v_k)\|<2^{-(n+1)}.
\]
\end{enumerate}
By the fact mentioned above, at least one of the two searches
succeeds, and the test will answer (a) or (b) according to the first
one succeeding.

If for some $i \leq n+1$ the answer is (a), we pick the least $i$
with this property and set $q(n+1) = i$, so that $R_{n+1} = R_n$ ($i
\neq N$ because $e(N) \in T_n$). Otherwise, if for all $i \leq n+1$
the answer is (b), then $q(n+1) = N$ and hence $R_{n+1} = R_n \cup
\{n+1\}$.

It is straightforward to check that $\set{(e \circ q) (j)}{j \in \N
\setminus R}$ is linearly independent and dense in $X$.
\end{proof}

The main feature of Lemma \ref{ueil} is that we can uniformly find a
sequence of linearly independent vectors whose linear span is dense
in $X$ by allowing repetitions of the single element $(e \circ q)
(0)$ and forgetting all occurrences of this element after the first.

\begin{definition}[$\BS^+$]
Let $\BS^+$ be the graph of the computable multi-valued function
$\zeta$ of Lemma \ref{ueil}. In other words,
\[
\BS^+ = \set{((X, \|\phantom{x}\|, e), q) \in \BS \times \Bai}{q \in \zeta (X, \|\phantom{x}\|, e)}.
\]
When we write $X^+ \in \BS^+$ we mean that $X \in \BS$ and $X^+ =
(X,q)$ for some $q \in \zeta (X)$.
\end{definition}

Using Lemma \ref{ueil} we obtain a uniform proof of Lemma 3 in
\cite{Bra08}.

\begin{definition}[Identity problem]
For an effective Banach space $(X, \|\phantom{x}\|, e)$ the identity
problem for $(X, \|\phantom{x}\|, e)$ is the set
\[
I (X, \|\phantom{x}\|, e) = \set{(s,t) \in \Seq \times \Seq}{a_e(s) = a_e(t)}.
\]
\end{definition}

\begin{lemma}[Identity problem lemma]\label{ip}
Given $((X, \|\phantom{x}\|, e), q) \in \BS^+$ let $e' = e \circ q$.
\begin{enumerate}[(1)]
\item The function $((X, \|\phantom{x}\|, e), q) \mapsto e'$ is
    computable;
\item $id: (X, \|\phantom{x}\|, e) \to (X, \|\phantom{x}\|, e')$
    and its inverse are uniformly computable in $((X,
    \|\phantom{x}\|, e), q) \in \BS^+$;
\item the function which associates to $((X, \|\phantom{x}\|,
    e), q) \in \BS^+$ the characteristic function of $I (X,
    \|\phantom{x}\|, e')$ is computable.
\end{enumerate}
\end{lemma}
\begin{proof}
(1) is obvious.

(2) For $id$ it is enough to show how to uniformly compute for any
$i \in \N$ a $p \in \Bai$ such that $((a_{e'} \circ p) (j))$ is a
Cauchy sequence converging effectively to $e(i)$. Let $R =
\set{j>0}{q(j) = q(0)}$ as in Lemma \ref{ueil}. The definition of
$p$ is by stages. Before stage $n$ we have defined $p[j_n]$ with
$j_n \leq n$ and at that stage we possibly define $p(j_n)$ as
follows. For each $s \leq n$ check whether
\[
\textstyle{d \left( e(i), \sum_{k < |s|, k \notin R} a_\Q (s(k))
\cdot e'(k) \right)_{[j_n+2]} < 2^{-(j_n+2)}} \tag{*}
\]
where, as in the proof of Theorem \ref{sel}, for $a \in \R$,
$a_{[k]}$ is a rational approximation within $2^{-k}$ of $a$. If (*)
holds for some $s \leq n$ let $p(j_n) = s$ (so that $j_{n+1} = j_n
+1$). If (*) fails for all $s \leq n$ do nothing, i.e.\ let $j_{n+1}
= j_n$. Since $e'$ is a fundamental sequence we have $\lim j_n =
\infty$, so that $p(j)$ is defined for every $j$. It is
straightforward to check that $p$ has the desired property.

The uniform computability of $id^{-1}$ is immediate, since $e'(n) =
e(m)$ whenever $q(n)=m$.

(3) Given $((X, \|\phantom{x}\|, e), q) \in \BS^+$ let $R^* = R \cup
\{0\} = \set{j}{q(j) = q(0)}$. To check whether $(s,t) \in I
(X,\|\phantom{x}\|,e')$ recall that $a_{e'}(s) = \sum_{i < |s|}
a_\Q(s(i)) \cdot (e \circ q) (i)$ and similarly for $a_{e'}(t)$.
Assuming $|s| \leq |t|$ we have that $a_{e'}(s) = a_{e'}(t)$ is
equivalent to the conjunction of the following conditions
\begin{itemize}
  \item $\fa i \leq |s| (i \notin R^* \to s(i)= t(i))$;
  \item $\fa i \leq |t| (i \geq |s| \land i \notin R^* \to
      a_\Q(t(i)) = 0)$;
  \item $\sum_{i\leq |s|, i \in R^*} a_\Q(s(i)) = \sum_{i \leq
      |t|, i \in R^*} a_\Q(t(i))$.
\end{itemize}
Since each of these conditions is computable in $R^*$, and hence in
$q$, this equivalence completes the proof.
\end{proof}

We now consider the space $\R^\N$ equipped with the (slightly
nonstandard) metric
\[
d((x_n), (y_n)) =
\sup \SET{\frac1{2^n} \frac{|x_n-y_n|}{1+|x_n-y_n|}}{n \in \N}.
\]
If $a(s) = (a_\Q(s(0)), \dots, a_\Q(s(|s|-1)), 0, 0 \dots)$ then it
is easy to check that $(\R^\N,d,a)$ is a computable metric space.

The main reason for using $d$ instead of the standard textbook
metric for $\R^\N$ (defined by a series rather than a $\sup$) is
that the open balls with respect to $d$ are of the form $I_0 \times
\dots \times I_n \times \R \times \R \times \cdots$, where each $I_i
\subseteq \R$ is an open interval. Of course, both metrics are
compatible with the product topology of $\R^\N$.

\begin{lemma}\label{product}
The function which maps every $(x_n) \in \R^\N$ to the compact space
$\prod_{n\in \N} [-|x_n|, |x_n|] \in \mcK (\R^\N)$ is computable.
\end{lemma}
\begin{proof}
The proof of this Lemma essentially consists in checking that the
proof of \cite[Lemma 4]{Bra08} is uniform. In doing so we spell out
a few more details of the proof.

To simplify the notation let $Y_{(x_n)} = \prod_{n\in \N} [-|x_n|,
|x_n|]$.

By \cite[Theorems 3.7, 3.8 and Proposition 4.2.2]{BP}, it suffices
to show that the function $(x_n) \mapsto Y_{(x_n)}$ is computable
when $Y_{(x_n)}$ is viewed as an element of $\mcA_+ (\R^\N)$ and of
$\mcK_- (\R^\N)$.

We first deal with $\mcA_+ (\R^\N)$. Define a computable $\r: \R^\N
\times \R^\N \to \R^\N$ by
\[
\r((x_n), (y_n)) = (\max \{-|x_n|, \min\{y_n,|x_n|\}\}).
\]
Then $\set{\r((x_n), a(s))}{s \in \Seq}$ is dense in $Y_{(x_n)}$.
This shows that the function mapping $(x_n) \in \R^\N$ to $Y_{(x_n)}
\in \mcA_+ (\R^\N)$ is computable.

To compute $Y_{(x_n)}$ as an element of $\mcK_- (\R^\N)$ we need to
show that we can enumerate a list of all finite coverings of
$Y_{(x_n)}$ consisting of rational open balls. In other words, we want
to show that the set of all finite sets $\{B_0, \dots, B_k\}$ of
rational open balls in $\R^\N$ such that $Y_{(x_n)} \subseteq
\bigcup_{i=0}^k B_i$ is r.e.\ in $(x_n)$. By our choice of the metric,
each $B_i$ is of the form $(\a_0^i, \b_0^i) \times \dots \times
(\a_{m_i}^i, \b_{m_i}^i) \times \R \times \R \times \cdots$ with $m_i
\in \N$ and $\a_0^i, \b_0^i, \dots, \a_{m_i}^i, \b_{m_i}^i \in \Q$. Let
$m = \max \set{m_i}{i \leq k}$. Now notice that $Y_{(x_n)} \nsubseteq
\bigcup_{i=0}^k B_i$ is equivalent to the existence of $\g_0, \dots,
\g_m \in \Q$ such that $\g_n \in \set{\a_n^i, \b_n^i}{i \leq k}$ for
each $n \leq m$ and
\[
(\g_0, \dots, \g_m, 0, 0, \dots) \in Y_{(x_n)} \setminus \bigcup_{i=0}^k B_i.
\]
Hence we need to check the co-r.e.\ in $(x_n)$ condition on finitely
many $(m+1)$-uples.
\end{proof}

We recall that the Banach-Alaoglu Theorem states that the closed
unit ball of the dual space of a normed vector space is compact in
the weak* topology. The next Theorem is a uniform version of Theorem
6 in \cite{Bra08}. The idea here is that we uniformly embed the
closed unit ball of the dual space of an element of \BS\ onto a
closed subset of a compact subset of $\R^\N$. Moreover this is done
taking into account the change of fundamental sequence provided by
Lemma \ref{ip}. In the statement of the Theorem the reader should
keep in mind that $\phi$ restricted to fixed $(X,q) \in \BS^+$ is
this embedding and $\chi$ computes its inverse, taking in input also
the norm of the functional.

\begin{theorem}[Uniform Computable Banach-Alaoglu Theorem]\label{BA}
Let $\phi: \sbsq \PF \times \Bai \to \R^\N$ be the function with
\[
\dom (\phi) = \set{(g_{(X,X,r)},q)}{r \leq 1 \land (X,q) \in \BS^+}
\]
defined by
\[
\phi (g_{(X,X,r)},q) = ((g \circ a_{e'}) (n)),
\]
where $e' = e \circ q$ as in Lemma \ref{ip}.
\begin{enumerate}[(1)]
\item $\phi$ is computable and $\phi (g_{(X,X,r)},q) = \phi
    (g'_{(X,X,r')},q)$ implies $g=g'$ and $r=r'$;
\item there exist computable functions $\widehat{\phantom{X}}:
    \BS^+ \to \mcK (\R^\N)$ and $\widetilde{\phantom{X}}: \BS^+
    \to \mcA_- (\R^\N)$ such that $\widetilde{X^+} \subseteq
    \widehat{X^+}$ and $\phi (g_{(X,X,r)},q) \in
    \widetilde{(X,q)}$;
\item there exists a computable $\chi: \sbsq \R^\N \times \BS^+
    \times \R \to \PF \times \Bai$ such that
      \[
\dom(\chi) = \SET{\left( (a_n), X^+, r \right)}{(a_n) \in \widetilde{X^+}
\land r = \sup \set{\tfrac{|a_n|}{\|a_{e'}(n)\|}}{n \in \N}},
      \]
      and we have always $\chi ((a_n), (X,q), r) = (g_{(X,X,r)},
      q)$ for some function $g$ such that $\phi(g_{(X,X,r)}, q) =
      (a_n)$.
\end{enumerate}
\end{theorem}
\begin{proof}
(1) is obvious.

(2) For $X^+ = (X,q)$ and $e' = e \circ q$, define
\[
\widehat{X^+} = \prod_{n \in \N} [-\|a_{e'}(n)\|,
\|a_{e'}(n)\|],
\]
and let $\widetilde{X^+}$ be the set of all $(a_n) \in \widehat{X^+}$
such that
\[
\forall \a, \b \in \Q\, \forall i,j,n \in \N (a_{e'}(n) = \a
a_{e'}(i) + \b a_{e'}(j) \implies a_n = \a a_i + \b a_j).
\]
By Lemma \ref{product} $\widehat{X^+} \in \mcK (\R^\N)$ and
$\widehat{\phantom{X}}$ is computable. To show that $\widetilde{X^+}
\in \mcA_- (\R^\N)$ notice that, given $\a$, $\b$, $i$, and $j$, we
can compute $k$ such that $\a a_{e'}(i) + \b a_{e'}(j) = a_{e'}(k)$.
Thus $a_{e'}(n) = \a a_{e'}(i) + \b a_{e'}(j)$ is equivalent to
$(n,k) \in I(X, \|\phantom{x}\|, e')$. By Lemma \ref{ip} we can
compute from $X^+$ the characteristic function of $I(X,
\|\phantom{x}\|, e')$ and thus check whether the latter condition
holds. It is now obvious that $\widetilde{X^+} \in \mcA_- (\R^\N)$
and that and that $\widetilde{\phantom{X}}$ is computable. It is
also obvious that $\phi (g_{(X,X,r)}, q) \in \widetilde{X^+}$.

(3) Let $(((a_n), X^+, r)) \in \dom(\chi)$ and notice that $r \leq
1$. We need to compute $g: X \to \R$ linear and bounded such that
$\|g\| = r$ and $g(a_{e'}(n)) = a_n$. Given $x \in X$ to compute
$g(x)$ within $2^{-k}$ it suffices to find $n$ such that $\|x -
a_{e'}(n)\| < 2^{-k}$. Then
\[
|g(x) - a_n| = |g(x) - g(a_{e'}(n))| \leq r \cdot \|x - a_{e'}(n)\| < 2^{-k}
\]
and we can use $a_n$ as an approximation of $g(x)$.
\end{proof}

The next Lemma is the uniform version of Theorem 5 of \cite{Bra08}.

\begin{lemma}\label{H}
Let $H: \sbsq \PF \times \Bai \to \mcA_- (\R^\N)$ be the function
with
\[
\dom(H) =\set{(f_{(X,A,r)}, q)}{(X,q) \in \BS^+ \land r=1}
\]
defined by
\[
H(f_{(X,A,1)}, q) = \set{\phi(g_{(X,X,1)}, q)}{g \restriction A = f}.
\]
Then $H$ is computable.
\end{lemma}
\begin{proof}
Given $(f_{(X,A,1)}, q) \in \dom(H)$ let $X^+ = (X,q)$. We can
compute $e' = e \circ q$ and $\set{y_i}{i \in \N}$, a dense subset
of $A \in \mcA_+ (X)$. Notice that $(a_n) \in H(f_{(X,A,1)}, q)$ if
and only if $(a_n) \in \widetilde{X^+}$ and
\[
\forall n,i\, |f(y_i) - a_n| \leq \|y_i - a_{e'}(n)\|.
\]
Therefore $H(f_{(X,A,1)}, q) \in \mcA_- (\R^\N)$. The computability
of $H$ is immediate.
\end{proof}

Finally we can prove the first half of our main result.

\begin{theorem}
\HB\ is \Sep-computable.
\end{theorem}
\begin{proof}
Given $f_{(X,A,1)} \in \PF$ by Lemma \ref{ueil} we can compute $q \in
\zeta(X)$, so that $X^+ = (X,q) \in \BS^+$. By Theorem \ref{BA} and
Lemma \ref{H} we can compute $\widehat{X^+} \in \mcK (\R^\N)$ and $C
= H (f_{(X,A,1)}, q) \in \mcA_- (\R^\N)$, so that $C \subseteq
\widehat{X^+}$. Notice that $C \neq \eps$ because the Hahn-Banach
Theorem holds. We can thus apply the multi-valued function
$\Sel_{\mcK(\R^\N)}$ defined in Theorem \ref{sel} to the pair
$(\widehat{X^+}, C)$ and select a point $(a_n) \in C$. Then
\[
\chi ((a_n), X^+, 1) = (g_{(X,X,1)}, q) \text{ for some }
g_{(X,X,1)} \in \HB(f_{(X,A,1)}).
\]

We have thus shown $\HB \cred \Sel_{\mcK(\R^\N)}$. Since
$\Sel_{\mcK(\R^\N)}$ is \Sep-computable by Theorem \ref{sel}.1, this
completes the proof.
\end{proof}

\subsection{Proof of $\Sep \cred \HB$}
The proof of the other half of our main result is obtained by
adapting the proof of Theorem IV.9.4 in \cite{sosoa}.

\begin{theorem}\label{reversal}
$\Sep \cred \HB$.
\end{theorem}
\begin{proof}
Let $p,q\in \Bai$ be such that $\ran(p)\cap \ran(q)=\eps$. We will
use $p$ and $q$ to compute $f_{(X,A,1)} \in \PF$ so that from any
element of $\HB(f_{(X,A,1)})$ we can compute an element of
$\Sep(p,q)$.

In particular, $X$ is a constructive Banach completion and,
following the construction in Subsection \ref{subs:cBc}, we need to
define a pseudo-norm on the set $C$ of all formal linear
combinations of elements of $\N$ with scalars in $\Q$. To define
this pseudo-norm (which depends on $p$ and $q$) we identify elements
of $\N$ with finite sequences of elements of $\Q^2$ as follows: $2n$
and $2n+1$ are identified respectively with the sequences $\la
(0,0), \dots, (0,0), (1,0) \ra$ and $\la (0,0), \dots, (0,0), (0,1)
\ra$ of length $n+1$. With this identification, $C$ is viewed as the
set $Q_2$ of all finite sequences of elements of $\Q^2$. We will
therefore define the pseudo-norm on $Q_2$.

Let
\[
\d_n= \left\{ \begin{array}{ll} 2^{-k} & \textrm{if $k=\m i(p(i)=n)$}\\
-2^{-k} & \textrm{if $k=\m i(q(i)=n)$}\\
0 & \textrm{otherwise.}
\end{array} \right.
\]
$\d_n$ is computable as a real number on the input $(p,q,n)$.

For $(\a,\b)\in \Q^2$, let
\[
\|(\a,\b)\|_n= \left\{ \begin{array}{ll} \max\{|\frac{1-\d_n}{1+\d_n}\a+\b|,|\a-\b|\} & \textrm{if $\d_n>0$}\\
 & \\
\max\{|\frac{1+\d_n}{1-\d_n}\a-\b|,|\a+\b|\} & \textrm{if $\d_n<0$}\\
 & \\
\max\{|\a+\b|,|\a-\b|\} & \textrm{if $\d_n=0$}\\
\end{array} \right.
\]
$\|(\a,\b)\|_n$ is computable as a real number on input
$(p,q,\a,\b,n)$. Notice that $\|(\a,0)\|_n = \|(0,\a)\|_n = |\a|$
for all $\a$ and $n$.

We can now define the pseudo-norm on $Q_2$ by
\[
\|\la (\a_i,\b_i)\ra_{i < k}\| = \sum_{i < k} 2^{-i-1} \cdot
\|(\a_i,\b_i)\|_i.
\]
This noted pseudo-normed space generates the constructive Banach
completion $X = X(p,q) \in \BS$\footnote{Intuitively, $X$ is the
$\ell_1$-sum of a sequence of $2$-dimensional Banach spaces with
slightly different metrics.}. As usual, we view $Q_2$ as a subset of
$X$.

Let
\[
A = \overline{\{\la (\a_i,0)\ra_{i < n}\}} \in \mcA_+ (X)
\]
and define $f: A \to \R$ by setting
\[
f(\la (\a_i,0)\ra_{i < n})=\sum_{i < n}2^{-i-1}\a_i
\]
and extending by continuity. The function $f$ is linear on $A$ and
is a bounded linear functional with $\|f\| \leq 1$, since
\begin{align*}
|f(\la (\a_i,0)\ra_{i < n})| & = |\sum_{i < n}2^{-i-1}\a_i|\\
 & \leq \sum_{i < n}2^{-i-1}|\a_i|\\
 & = \sum_{i < n}2^{-i-1}\|(\a_i,0)\|_i\\
 & = \|\la (\a_i,0)\ra_{i < n}\|.
\end{align*}
Moreover $\|\la (2,0)\ra \|=1$ and $f(\la (2,0)\ra )=1$, which shows
that $\|f\|= 1$. By evaluation and type conversion, one can compute
a realizer of $f$.

Therefore $f_{(X,A,1)} \in \PF$ has been computed from $(p,q)$ and
moreover we have $f_{(X,A,1)} \in \dom (\HB)$. Applying $\HB$ we
obtain $g_{(X,X,1)} \in \PF$ with $g \restriction A = f$.

For any $n \in \N$ let $z_n \in Q_2$ be the sequence $\la (0,0),
\dots, (0,0), (0,1) \ra$ of length $n+1$. Then $|g(z_n)| \leq
\|z_n\| = 2^{-n-1}$.

If $n \in \ran(p)$ then $\d_n>0$ and notice that, for $w_n = \la
(0,0), \dots, (0,0), (1+\d_n,0)\ra$ of length $n+1$ we have
\begin{align*}
|f(w_n)+\d_n g(z_n)| & = |g(w_n) + \d_n g(z_n)|\\
& = |g(w_n+\d_n z_n)|\\
& \leq \|w_n+\d_n z_n\|\\
& = \|\la (0,0), \dots, (0,0), (1+\d_n,\d_n)\ra\|.
\end{align*}
Since $\d_n>0$
\[
\|(1+\d_n,\d_n)\|_n=\max\{|\tfrac{(1-\d_n)}{(1+\d_n)}(1+\d_n)+\d_n|,|1+\d_n-\d_n|\}=1,
\]
and so $\|\la (0,0), \dots, (0,0), (1+\d_n,\d_n)\ra \|=2^{-n-1}$. We
deduce that $|2^{-n-1} (1 + \d_n) + \d_n g(z_n)| = |2^{-n-1} + \d_n
(2^{-n-1} + g(z_n))| \leq 2^{-n-1}$. Therefore $\d_n (2^{-n-1} +g
(z_n)) \leq 0$ and so $g(z_n) \leq -2^{-n-1}$. Since $|g(z_n)| \leq
2^{-n-1}$ then $g(z_n) = -2^{-n-1}$.

Similarly, if $n \in \ran(q)$ (and thus $\d_n<0$) we obtain $g(z_n)
= 2^{-n-1}$ by considering
\[
|2^{-n-1} (1-\d_n) + \d_n g(z_n)| = |2^{-n-1} + \d_n(g(z_n) - 2^{-n-1})| \leq 2^{-n-1}.
\]

To compute an element of $\Sep(p,q)$, given $n$ look for the
approximation of $g(z_n)$ within $2^{-n-2}$ and check if it is
positive or not. This shows that from any $g$ such that $g_{(X,X,1)}
\in \HB(f_{(X,A,1)})$ we can uniformly compute an element of
$\Sep(p,q)$.
\end{proof}

\bibliography{HahnBanach}
\bibliographystyle{alpha}

\end{document}